\documentclass{preprint}
\usepackage[latin1]{inputenc}
\usepackage{amsmath}
\usepackage{amsfonts}
\usepackage{times}
\usepackage{mhequ}
\usepackage{mhenvs}
\usepackage{Symbols}
\usepackage{ScriptFonts}

\def\D{\mathbf{D}}
\def\E{\mathbb{E}}
\def\P{\mathbb{P}}
\def\W{\mathbf{W}}
\def \R{\mathbb{R}}

\def\TV{{\mathrm TV}}
\def\mn{{\mathrm{min}}}
\def \A{\mathcal{A}}
\def\Q{\bar \CQ}
\def\CG{\mathcal{G}}
\def\eps{\varepsilon}
\newcommand{\tE}{{\tilde{\mathbb{E}}}} %macro for conditional expectation
\newcommand{\tprob}{{\tilde{\mathbb{P}}}}
\newcommand{\tp}{\mathcscr{E}(x,\omega)} % macro for the poly  exp bound
\newcommand{\K}{\mathcscr{K}} % macro for the poly  exp bound

\newcommand{\bth}{B_t}

\newtheorem{assumption}[lemma]{Assumption}
\def\CWt{\widetilde\CW}
\def\I{\CI}

\def\noise{B}
\def\y{y}
\def\z{z}
\def\hz{Z}

\definecolor{darkred}{rgb}{0.9,0.1,0.1}

\begin{document}

%\date{}
\title{Ergodicity of hypoelliptic SDEs driven by fractional Brownian motion}

\author{M. Hairer\inst{1,2}, N.S. Pillai\inst{3}}
\institute{Mathematics Institute, The University of Warwick\\
 \email{M.Hairer@Warwick.ac.uk}
\and
Courant Institute, New York University\\
 \email{mhairer@cims.nyu.edu}
\and
CRiSM, The University of Warwick\\
\email{N.Pillai@Warwick.ac.uk}}
\titleindent=0.55cm

\maketitle

\begin{abstract}
We demonstrate that stochastic differential equations (SDEs) driven by fractional
Brownian motion with Hurst parameter $H > {1\over 2}$ have similar ergodic properties
as SDEs driven by standard Brownian motion. The focus in this article is on hypoelliptic
systems satisfying H\"ormander's condition. We show that such systems enjoy a suitable version
of the strong Feller property and we conclude that under a standard controllability condition
they admit a unique stationary solution that is physical in the sense that it does not 
``look into the future''.

The main technical result required for the analysis is a bound on the moments 
of the inverse of the Malliavin covariance matrix, conditional on the past of the driving noise.
\end{abstract}

\renewcommand\abstractname{R\'esum\'e}
\vspace{1em}
\begin{abstract}
Nous d\'emontrons que les \'equations diff\'erentielles stochastiques (EDS) conduites
par des mouvements Browniens fractionnaires \`a param\`etre de Hurst $H > {1\over 2}$
ont des propri\'et\'es ergodiques similaires aux EDS usuelles conduites par des 
mouvements Browniens.  L'int\'er\^et principal du pr\'esent article est de pouvoir tra\^\i ter \'egalement
des syst\`emes hypoelliptiques satisfaisant la condition de H\"ormander. Nous montrons
qu'une version ad\'equate de la propri\'et\'e de Feller forte est v\'erifi\'ee par de tels syst\`emes
et nous en d\'eduisons que, sous une propri\'et\'e de controllabilit\'e usuelle, ils admettent
une unique solution stationnaire qui ait un sens physique.

L'ingr\'edient principal de notre analyse est une borne sup\'erieure sur les moments inverses
de la matrice de Malliavin associ\'ee, conditionn\'ee au pass\'e du bruit.
\\[0.5em]
\end{abstract}

\section{Introduction}
In this paper we study the ergodic properties of stochastic differential
equations (SDEs) of the type
\begin{equ}[e:SDEmain]
dX_t = V_0(X_t)\,dt + \sum_{i=1}^d V_i(X_t)\circ dB^i_t\;,\quad X_0\in \R^n\;,
\end{equ}
driven by independent fractional Brownian motions (fBm) $B^i$ with fixed Hurst index
$H \in (1/2,1)$. Recall that fBm is the centred Gaussian process with $B^i_0 = 0$ and
\begin{equ}
\E |B^i_t - B^i_s|^2 = |t-s|^{2H}\;.
\end{equ}
Although the basic theory of SDEs driven by fBm (for $H \in (1/4, 1)$)
is now well established (see for example \cite{Nual:06,Mishura,CoutQian:02,Friz}), the ergodicity of such SDEs
has only been studied recently. The main difficulty is that non-overlapping increments of fBm are not independent, so 
that we are dealing with processes that do not have the Markov property.
As a consequence, the traditional ergodic theory for Markov processes as in \cite{MeynTweedie}
does not apply to this situation.

To our knowledge, the first result on the ergodicity of \eref{e:SDEmain} was given in \cite{Hair:05}, where the author
considers the case where the vector fields $V_i$ for $i > 0$ are \textit{constant} and \textit{non-degenerate}
 in the sense
that the corresponding matrix is of rank $n$. In this case, an explicit coupling argument allows to show
that any solution converges to the unique stationary solution at speed bounded above by 
approximately $t^{-1/8}$.
The argument however relies in an essential way on the additivity of the noise.

A more general construction was given in \cite{Hair:Ohas:07,Hair:09}, where the authors built a general
framework of stochastic dynamical systems (SDSs) very similar to that of random dynamical systems (RDSs) 
\cite{Arnold}. In this framework, the notion of an ``invariant measure'' can be defined for an autonomous equation
driven by an ergodic process with stationary but not necessarily independent increments. 
Besides technical questions of continuity, the main difference between the framework of SDSs and that of RDSs is one of viewpoint. An RDS is considered as
a dynamical system on the state space times the space of all `futures' of the driving process, endowed with a cocycle structure. 
An SDS on the other hand is considered as a Markov process on the state space times the space of all `pasts' of the driving
process, endowed with a cocycle structure. This shift of viewpoint has two important advantages when studying the ergodic 
properties of such systems:
\begin{claim}
\item Since the `future' of the driving noise is not part of the augmented phase space, we do not have to deal with `unphysical' 
stationary solutions, where the value of the process at a given time is a function of the future realisation of the driving noise.
\item It allows to build our intuition based on the theory of Markov processes, rather than the theory of dynamical systems.
There, many criteria for the uniqueness of an invariant measure are known to exist.
\end{claim}
Inspired by the celebrated Doob-Khasminskii theorem (see for example \cite{DPZ}), \cite{Hair:Ohas:07,Hair:09}
introduced a notion of a strong Feller property for SDSs that is a natural
generalisation of the same notion for Markov semigroups. It turns out that for a large class of `quasi-Markovian'
SDSs (this includes SDEs of the type \eref{e:SDEmain}, but also many other examples), uniqueness of the invariant measure
then follows like in the Markovian case from the strong Feller property, combined with some type of topological
irreducibility. To illustrate these results, denote by $A_{t,x}$ the closure of the set of all points that are accessible at time $t$
for solutions to \eref{e:SDEmain} starting at $x$ for some smooth control $B$. Then, one has

\begin{theorem} \label{thm:infexun}
If there exists $t > 0$ such that \eref{e:SDEmain} is strong Feller at time $t$
and such that $\bigcap_{x \in \R^n} A_{t,x} \neq \emptyset$, then \eref{e:SDEmain} can have at most one
invariant measure.
\end{theorem}

\begin{remark}
Since \eref{e:SDEmain} does not determine a Markov process, it is not clear \textit{a priori} what is the correct 
notion of an invariant measure. The notion retained here is the one introduced in \cite{Hair:Ohas:07,Hair:09}. Essentially,
an invariant measure for \eref{e:SDEmain} is the law of a stationary solution which does not `look into the future'.
Similarly, it is not clear what is the correct notion of a `strong Feller property' for such a system. Precise definitions
can be found in Section~\ref{sec:basicdef} below.
\end{remark}

Furthermore, it was shown in \cite{Hair:Ohas:07} that if \eref{e:SDEmain} is `elliptic' in the sense that the matrix determined by $\{V_i(x)\}_{i>0}$ is 
of full rank at every $x\in \R^n$, then it does indeed satisfy the strong Feller property.
In the case of SDEs driven by white noise, it is known on the other hand that the strong Feller property holds whenever
the vector fields $V_i$ satisfy H\"ormander's bracket condition, namely that the Lie algebra generated by
$\{\d_t + V_0, V_1,\ldots, V_d\}$ spans $\R^{n+1}$ (the additional coordinate corresponds to the $\d_t$ direction) at every point. See also Assumption~\ref{ass:Hormander} below for a precise formulation.

Note at this stage that the strong Feller property of \eref{e:SDEmain} is closely related to the existence of densities
for the laws of its solutions, but cannot be deduced from it directly. The existence of smooth densities under an ellipticity assumption
was shown in \cite{Hu:Nual:07,Nour:Simo:06}, while it was shown under the assumption of H\"ormander's bracket condition in
 \cite{Baud:Hair:07}. In this paper, we are going to address the question of showing that the strong Feller property
 (in the sense given in \cite{Hair:Ohas:07,Hair:09}, see also Definition~\ref{sfellerdefinition} below) holds for \eref{e:SDEmain}
 under H\"ormander's bracket condition.
 
 Our main result is:
 
\begin{theorem}\label{thm:main}
Assume that the vector fields $\{V_i\}$ have bounded derivatives of all orders for $i \ge 0$ and
 furthermore are bounded for $i \ge 1$. Then, if H\"ormander's bracket condition holds at every point, \eref{e:SDEmain} is strong Feller.
In particular, if there exists $t>0$ such that $\bigcap_{x \in \R^n} A_{t,x} \neq \emptyset$ and if there exist constants $M_1, M_2>0$ such that
\begin{equ}[e:dissip]
\scal{x, V_0(x)} \le M_1 - M_2 |x|^2\;,
\end{equ}
then \eref{e:SDEmain} admits exactly one invariant measure.
\end{theorem}

\begin{remark}
Since we would like to satisfy \eref{e:dissip} in order to guarantee the existence of an invariant measure by \cite[Prop~4.6]{Hair:Ohas:07},
we do not impose that $V_0$ is bounded. This creates some technical difficulties in obtaining a priori bounds on
the solutions which are not present in \cite{Baud:Hair:07}.
\end{remark}

\begin{remark}
The notion of an `invariant measure' studied in \cite{Bau:Cou:07} is much more restrictive than
the one studied here, since they considered the situation where the initial condition of the equations is
\textit{independent} of the driving noise, whereas we consider the case where interdependence with 
the past of the driving noise is allowed. In particular, even the case $n=d=1$
with $V_0(x) = -x$ and $V_1(x) = 1$ admits no invariant measure in the sense of \cite{Bau:Cou:07}.

On the other hand, our definition of an invariant measure is more restrictive than that
of a random invariant measure in \cite{Arnold}. Since it generalises the notion of an invariant measure for a Markov process,
this is to be expected. Indeed, there are simple examples of elliptic diffusions on the circle
whose Markov semigroup admits a unique invariant measure, but that admit more than one random invariant measure when
viewed as RDS's. 
\end{remark}

The remainder of the article is structured as follows. In Section~\ref{sec:prelim}, we set up our notations and recall the relevant
results from \cite{Hair:05} and \cite{Hair:Ohas:07}. In Section~\ref{sec:bounds}, we derive the 
necessary moment estimates for
the solutions of the SDE \eref{e:SDEmain}. In Section~\ref{sec:smooth}, we then obtain an invertibility result on the conditioned
Malliavin covariance matrix of the process \eref{e:SDEmain}, provided that  H{\"o}rmander's condition holds. 
Similarly to \cite{Baud:Hair:07}, this allows us to show the smoothness of the laws
of solutions to \eref{e:SDEmain}, conditional on the past of the driving noise. 
In Section~\ref{sec:SF}, we finally show that the strong Feller property is satisfied for \eref{e:SDEmain} under H{\"o}rmander's
condition, thereby concluding the proof of Theorem \ref{thm:main}.
 
\subsection*{Acknowledgements}

{\small
We would like to thank the anonymous referees for carefully reading a preliminary version of this article and 
encouraging us to improve the presentation considerably. NP would like to thank the department of statistics
of Warwick university for support through a CRiSM postdoctoral fellowship, as well as the Courant institute at NYU
where part of this work was completed.
The research of MH was supported by 
an EPSRC Advanced Research Fellowship (grant number EP/D071593/1) and a Royal Society Wolfson Research Merit Award.
}
 
\section{Preliminaries}
\label{sec:prelim}

In this section, we describe the general framework in which we view solutions to \eref{e:SDEmain} and we recall some of the
basic results from \cite{Hair:05,Hair:Ohas:07,Hair:09}.

\subsection{General framework}
\label{sec:basicdef}
Let $\CC^{\infty}_0(\mathbb{R}_-;\R^d)$ be the set of smooth compactly supported functions on $\R_-$
which take the value $0$ at the origin.
 For  $\gamma , \delta \in (0,1)$  define 
 $\CW_{(\gamma,\delta)}$ to be the completion of 
 $\CC^{\infty}_0 (\mathbb{R}_-;\R^d)$ 
 with respect to the norm
\begin{equ}
\label{noisespace}
\|\omega\|_{(\gamma,\delta)} \equiv \sup_{\substack{s,t \in \mathbb{R}_- \\ s \neq t}} {|\omega(t) - \omega(s)|
\over |t-s|^\gamma (1 + |t| + |s|)^\delta}\;.
\end{equ}
We write $\widetilde{\CW}_{(\gamma,\delta)}$ for the corresponding space
containing functions on $\R_+$ instead of $\R_{-}$.
Note that when restricted to a compact time interval, the norm \eref{noisespace} 
is equivalent to the usual H\"older norm with exponent $\gamma$. 
Moreover, $\CW_{(\gamma,\delta)}$ is
a separable Banach space. 

For $H \in (1/2,1)$, $\gamma \in (1/2, H)$ and $\gamma + \delta \in (H,1)$, 
it can be shown that there 
 exists a Borel probability measure $\mathbb{P}_H$ on $\Omega = \CW_{(\gamma,\delta)}\times \widetilde \CW_{(\gamma,\delta)}$
such that the canonical process associated to $\mathbb{P}_H$ is a two-sided
fractional Brownian motion with Hurst parameter $H$ \cite{Hair:05, Hair:Ohas:07}.
We fix such values $\gamma$, $\delta$, and $H$ once and for all and we drop the subscripts $(\gamma,\delta)$ from the
spaces $\CW$ and $\CWt$ from now on.
Note that $\P_H$ is a product measure if and only if $H = 1/2$, but in general we can disintegrate it into a
projected measure (denoted again by $\P_H$) on $\CW$ and regular conditional probabilities $\CP(\omega, \cdot)$ on $\CWt$.
Since $\P_H$ is Gaussian, there is a unique choice of $\CP(\omega, \cdot)$ which is weakly continuous in $\omega$.

We denote by $\phi \colon \R^n \times \CWt \to \CC(\R_+,\R^n)$ the solution map to \eref{e:SDEmain}. It is well-known
\cite{Youn:36,Hu:Nual:07} that if the vector fields $V_i$ have bounded derivatives, then this solution is well-defined in a pathwise sense
(integrals are simply Riemann-Stieltjes integrals) for all times. We also define the shift map $\theta_t \colon \CW \times \CWt \to \CW \times \CWt$ by
identifying elements $(\omega, \tilde \omega) \in \CW \times \CWt$ with the corresponding `concatenated' function $(\omega \sqcup \tilde \omega) \colon \R \to \R^d$ and
setting
\begin{equ}
\theta_t (\omega,\tilde \omega) \sim (\omega \sqcup \tilde\omega)(t+\cdot) - (\omega \sqcup \tilde\omega)(\cdot)\;.
\end{equ}
Denote $\phi_t (x, \tilde{\omega}) = \phi(x,\tilde{\omega})(t)$ and denote by $\Phi_t\colon \R^n \times \CW \times \tilde \CW \to \R^n \times \CW$ the augmented solution map given by
\begin{equ}
\Phi_t(x,\omega,\tilde\omega) = \bigl(\phi_t(x,\tilde \omega), \Pi_{\CW}\theta_t(\omega,\tilde\omega)\bigr)\;,
\end{equ}
where $\Pi_\CW$ is the projection onto the $\CW$-component. For a measurable function $f: \mathcal{X} \mapsto \mathcal{Y}$
between two measurable space $\mathcal{X}, \mathcal{Y}$ and a measure $\mu$ on $\mathcal{X}$ we define
the push forward measure $f^* \mu = \mu \circ f^{-1}$.
With this notation, we can view the solutions to \eref{e:SDEmain} as a Markov process on $\R^n \times \CW$ with transition
probabilities given by
\begin{equ}
\CQ_t(x,\omega;\cdot\,)  = \Phi_t(x,\omega,\cdot\,)^* \CP(\omega,\cdot\,)\;.
\end{equ}
These transition probabilities can actually be shown to be Feller \cite{Hair:Ohas:07}, but they are certainly \textit{not} strong Feller
in the usual sense, since transition probabilities starting from different instances of $\omega$ remain mutually singular for all times. 
Instead, we will use a notion of strong Feller property that is better suited for the particular structure of the problem at hand, see 
Definition~\ref{sfellerdefinition} below. 

However, the question of uniqueness of the invariant measure for \eref{e:SDEmain} should not be interpreted as the question of uniqueness
of the invariant measure for $\CQ_t$. This is because one might imagine that the augmented phase space $\R^n \times \CW$ contains some
`redundant' randomness that is not used to describe the stationary solutions to \eref{e:SDEmain}. (This would be the case for example
if the $V_i$'s are not always linearly independent.) One would like therefore to have a concept of uniqueness for the invariant measure that is independent of the 
particular description of the driving noise.

To this end, we introduce the Markov transition kernel  $\bar \CQ$ from $\R^n \times \CW$ to $\CC(\R_+,\R^n)$ given by
\begin{equ}
\bar \CQ(x,\omega;\cdot\,)  = \phi(x,\cdot)^* \CP(\omega,\cdot)\;.
\end{equ}
This is the conditional law of the solution to \eref{e:SDEmain} given a realisation $\omega$ of the past of the driving noise. 
The action of the Markov transition kernel $\bar \CQ$ on a  measure $\mu$ on $\R^n \times \CW$ is given by
\begin{equ}
\bar \CQ \mu \,(A) = \int_{\R^n \times \CW} \bar \CQ(x,\omega; A) \, \mu(dx , d\omega)\;.
\end{equ}
With this notation, we have a natural equivalence relation between measures on $\R^n \times \CW$ given by
\begin{equ}
\mu \sim \nu \qquad\Leftrightarrow\qquad \bar \CQ \mu = \bar \CQ \nu\;.
\end{equ}
With these definitions at hand, the statement `the invariant measure for \eref{e:SDEmain} is unique' should be
interpreted as `the Markov semigroup $\CQ_t$ has a unique invariant measure, modulo the equivalence relation $\sim$'.

\subsection{Ergodicity of SDEs driven by fBm}
\label{sec:ergodic}

We now summarise some of the relevant results from \cite{Hair:05,Hair:Ohas:07,Hair:09}
giving conditions for the uniqueness of the invariant measure for \eref{e:SDEmain}.
This requires a notion of `strong Feller' property for \eref{e:SDEmain}. We stress again that the
definition given here has \textit{nothing to do} with the strong Feller property of $\CQ_t$. It rather
generalises the notion of the strong Feller property for the Markov process associated to \eref{e:SDEmain}
in the case where the driving noise is white in time.

Let $R_t\colon \CC(\mathbb{R}_+,\CX) \mapsto \CC([t,\infty),\CX)$ 
denote the natural  restriction map and let $\|\cdot\|_\TV$ denote the total variation norm.
We then say that:

\begin{definition} \label{sfellerdefinition}
The solutions to \eref{e:SDEmain} are said to be \textit{strong Feller} at time $t$ if there 
exists a jointly
continuous function $\ell:(\R^n)^2 \times \CW \rightarrow \R_+$ such that
\begin{equ}[eqn:SF]
\|R^*_t\Q(x,\omega;\cdot) - R^*_t\Q(y,\omega;\cdot\,) \|_\TV \le
\ell(x,y,\omega)\;,
\end{equ}
and $\ell(x,x,\omega)=0$ for every $x\in \R^n$ and every $\omega \in
\CW$.
\end{definition}

We also introduce the following notion of irreducibility:

\begin{definition} \label{irreducibility}
The solutions to \eref{e:SDEmain} are said to be \textit{topologically irreducible} at time $t$
if, for every $x \in \R^n$, $\omega \in \CW$ and every non-empty open set 
$U\subset \R^n$, one has $\CQ_t(x,\omega;  U \times \CW ) > 0$.
\end{definition}

The following result from \cite[Thm 3.10]{Hair:Ohas:07} is the main abstract uniqueness
result on which this article is based:
\begin{theorem}\label{doob}
If there exist times $s>0$ and $t>0$ such that the solutions to \eref{e:SDEmain} are strong Feller 
at time $t$ and irreducible at time $s$, then \eref{e:SDEmain}
can have at most one invariant measure.
\end{theorem}
%The irreducibility of the SDE \eref{e:SDEmain}, with $V_0$ satisfying 
%\eref{e:dissip} was establised in \cite{Hair:Ohas:07}. 
Therefore, in order to prove Theorem~\ref{thm:main}, the only missing ingredient that we need to establish
is the strong Feller property for \eref{e:SDEmain} under H\"ormander's bracket condition. 
 
 \subsection{Notations and definitions}

For $T > 0$ and measurable $f:[0,T] \mapsto \R^n$, set
\begin{equ}
\|f\|_{0,T,\infty} = \sup_{t \in [0,T]} |f(t)|\;, \qquad
\|f\|_{0,T,\gamma} = \sup_{s,t \in [0,T]} \frac{|f(t) - f(s)|}{|t-s|^\gamma}\;.  \label{eqn:Holdnorm}
\end{equ}
For $\alpha \in (0,1)$, we also define the fractional
integration operator $\CI^{\alpha}$ and the 
corresponding fractional differentiation operator
$\CD^\alpha$ by
\begin{equs}[e:frac]
\CI^{\alpha}f(t) &\equiv{1 \over \Gamma(\alpha)}\int_0^t (t-s)^{\alpha-1}f(s)\,ds \;,\\
\CD^{\alpha}f(t) &\equiv{1 \over \Gamma(1-\alpha)}{d \over dt} \int_0^t (t-s)^{-\alpha}f(s)\,ds \;.
\end{equs}

\begin{remark}\label{rem:idinvers}
The operators $\CI^\alpha$ and $\CD^\alpha$ are inverses of each other,
see \cite{Sam:Kil:Mar:93} for a survey of fractional integral operators. 
\end{remark}

The reason why these operators
are crucial for our analysis is that the Markov transition kernel $\CP$ appearing in Section~\ref{sec:basicdef}
is given by translates of the image of Wiener measure under $\CI^{H-1/2}$, see also Lemma~\ref{disintegrationformula} below.

Indeed, by the celebrated Mandelbrot-Van Ness
representation of the fBm \cite{Mand:Vann:68}, we may express the
two-sided fBm $B$
with Hurst parameter $H \in (0,1)$ in terms of a two-sided standard Brownian motion $W$ as
\begin{equ}
\label{repr1}
B_t=\alpha_H \int_{-\infty}^0(-r)^{H-1/2}(dW_{r+t} - dW_r)\;
\end{equ}
for some $\alpha_H > 0$.
The advantage of this representation
is that it is invariant under time-shifts, so that
it is natural for the study of ergodic properties, see \cite{Samo:Taqq:94} for more details.

Define now the operator $\CG\colon \CW \to \CWt$ by:
\begin{equ}[eqn:opera]
\CG\omega(t) \equiv  \gamma_H\int_0^{\infty} {1\over r }g\Big({t\over
r}\Big)\omega(-r)\,dr \;,
\end{equ}
where the kernel $g$ is given by
\begin{equs}[eqn:gfun]
 g(x) \equiv \,x^{H-1/2} + (H-3/2)\,x\,\int_0^1 {(u+x)^{H-5/2}\over
(1-u)^{H-1/2}}\,du\;,
\end{equs}
and the constant $\gamma_H$ is given by $\gamma_H = (H-1/2)\alpha_H \alpha _{1-H}$. Here, $\alpha_H$ is
the constant appearing in \eref{repr1}.
It was shown in \cite[Prop~A.2]{Hair:Ohas:07}  that $\CG$ is indeed
 a bounded linear operator from $\CW$ into 
$\CWt$. Furthermore, we can quantify its behaviour near $t=0$ in the following way:
%
%Define the (random) time derivative of the conditional mean process $c(t): \R \mapsto \R^n$:
%\begin{equs} \label{eqn:ct}
%c(t) = \frac{d}{dt} m_t
%\end{equs}
%The following lemma shows that the function $t \mapsto m_t$ is smooth
%on any interval bounded away from $0$, and its derivative $c(t)$
%has a divergence of $\mathcal{O}(t^{H-1})$
%near $0$. This is quite expected since $m_t$ is just the conditional mean of the
%fBM and thus has a regularity of $\mathcal{O}(t^H)$ at $t$ near $0$. This is made precise in the following result,
%the proof of which is postponed to the appendix:

\begin{lemma}\label{lem:divct} 
On
any time interval bounded away from $0$, the map
$t \mapsto \CG \omega(t)$  is $\CC^\infty$.
Furthermore, if we set $f_\omega (t) = t {d\over dt} \CG \omega(t)$, then we have $f_\omega(0) = 0$ and for every $T>0$, there exists
a constant $M_T$ such that 
$\|f_\omega \|_{0,T,\gamma} < M_T\|\omega\|_{(\gamma,\delta)}$.
\end{lemma}

The proof of this result is postponed to the appendix.

\begin{remark}
Since $\P_H$ is Gaussian, it follows in particular that both $\|\CG \omega\|_{0,T,\gamma}$ and 
$\|f_\omega\|_{0,T,\gamma}$ have exponential moments
under $\P_H$ by Fernique's theorem \cite{Bogachev}.
\end{remark}

Let $\tau_h \colon w \mapsto w + h$ denote the translation map on 
$\CWt$. We cite the following result from \cite[Lemma 4.2]{Hair:Ohas:07}:
\begin{lemma}\label{disintegrationformula}
The regular conditional probabilities $\CP(\omega, \cdot)$ of $\P_H$ given $\omega \in \CW$ are
\begin{equ}
\CP(w,\cdot\,) = \bigl(\tau_{\CG w}\circ \CI^{H - {1\over 2}}\bigr)^* \W\;,
\end{equ}
where $\W$ is the standard $d$-dimensional Wiener measure over $\mathbb{R}_+$ and $\CI^\alpha$ is
as defined in \eref{e:frac}. 
\end{lemma}

We will henceforth interpret the above lemma in the following way. The driving 
fractional Brownian motion $B$ can be written as the sum of two independent processes: 
\begin{equ}
B_t  \stackrel{\mathrm{d}}{=} \tilde{B}_t + (\CG \omega)(t) = \tilde B_t + m_t\;,
\quad \tilde{B}_t = \alpha_H \int_0^t (t-s)^{H-1/2}\,dW(s)\;, \label{e:defBtilde}
\end{equ}
where  $W$ is a standard Wiener process independent of the `past' $\omega \in \CW$.
This notation will be used repeatedly in the sequel.

\begin{remark}\label{rem:defEtilde}
In the notation of Section~\ref{sec:basicdef}, we will repeatedly use the notations $\tE$ (and $\tprob$) for 
conditional expectations (and probabilities) over $\tilde \omega$ with $\omega$ fixed. 
In the notation \eref{e:defBtilde}, this is the same as fixing $\omega$ and taking expectations with respect to $\tilde B$.
\end{remark}

\section{Estimates on the solutions}
\label{sec:bounds}

In this section, we derive moment estimates for the solutions to equations of the form:
\begin{equ}[e:SDE]
dX_t =  V_0(X_t) \,dt + \sum_{i=1}^d  V_i(X_t)\circ d\bth^i\;,\quad t \geq 0\;,
\end{equ}
where $x = X_0 \in \mathbb{R}^n$ and $V_i, i = 1, 2,\cdots,d$ are bounded,
 $\CC^\infty$ vector fields in 
$\mathbb{R}^n$ with bounded derivatives, and $V_0$  is a possibly \emph{unbounded} $\CC^\infty$ vector field
with bounded derivatives. Our bounds will be purely pathwise bounds, so the fact that the $B^i$'s are sample paths
of fractional Brownian motion is irrelevant, except to get a bound on their H\"older regularity.

All that we will assume is that the driving process is $\gamma$-H\"older continuous for some $\gamma > 1/2$. 
The main result in this section are a priori bounds not only on \eref{e:SDE}, but also on its Jacobian and its second variation
with respect to the initial condition. This will be a slight
generalisation of the results in \cite{Hu:Nual:07}, which required the drift term $V_0$ to be bounded.

 Let $\phi_t(x,\omega) \equiv X_t$ denote the flow map solving \eref{e:SDE} with initial condition $X_0 = x$ 
 and define its Jacobian by
 \begin{equs}
J_{0,t} &\equiv \frac{\partial{\phi_t(x,\omega)}}{\partial x }\;.
\end{equs}
For notational convenience,  set $V = (V_1,V_2, \cdots,  V_d)$.
Then, we can write \eref{e:SDE} in compact form:
\minilab{e:sdejac}
\begin{equs}
dX_t  &= V_0(X_t) \,dt + V(X_t) \circ dB_t \;,  \label{eqn:sdedem}\\
dJ_{0,t} &= V'_0(X_t)J_{0,t} dt + V'(X_t)J_{0,t} \circ dB_t \;,\label{eqn:sflow}\\
dJ^{-1}_{0,t} &= -J^{-1}_{0,t}V'_0(X_t) \,dt  - J_{0,t}^{-1}V'(X_t) \circ dB_t\;. \label{eqn:sflowinv}
\end{equs}
Here, both $J$ and $J^{-1}$ are $n\times n$ matrices, and $J_{0,0} = J^{-1}_{0,0} = 1$.
One crucial ingredient in order to obtain the bounds on the Malliavin matrix required to show the strong Feller property
is control on the moments of both the solution and its Jacobian. These bounds 
do not quite follow from standard results
since most of them require that $V_0$ is also bounded. 
However  we cannot assume this  since we need condition \eref{e:dissip} 
 for showing ergodicity.
\subsection{Moments estimates}

Let  $\CL_k(\CX_1,\ldots,\CX_k;\CY)$ denote the set of $k$-multilinear maps
from $\CX_1\times\ldots\times \CX_k$ to $\CY$. As usual, $\CL_1$ is denoted by $\CL$.
In this section, we consider a system of differential equations of the form
\minilab{e:generalequ}
\begin{equs}
x_t &= x_0 + \int_{0}^t f_0(x_r)\,dr + \int_0^t f(x_r)\,d\noise_r\;,  \label{eqn:nuaxt0}\\
\y_t &= \y_0 + \int_0^t g(x_r)(\y_r, \,d \noise_r)\;, \label{eqn:nuazt0}
\end{equs}
where $x_t \in \R^n$, $\y_t \in \R^m$, $\noise:[0,\infty) \mapsto \mathbb{R}^d$ is a H\"{o}lder function of order $\gamma > 1/2$ and
$f_0 \colon \R^n \to \R^n$, $f:\R^{n} \mapsto \CL(\mathbb{R}^{d}; \R^n)$, $g:\R^n \mapsto \CL_2(\R^m, \R^{d}; \R^{m})$  are given $\CC^1$ functions.
Note that \eref{e:sdejac} is indeed of this form with $m = 2n^2$.

\begin{remark}
The reason why we do not include a `$dr$' term in \eref{eqn:nuazt0} is because this is already covered by the present
result by setting $\noise_1(r) = r$ for example. Treating this term separately like in \eref{eqn:nuaxt0} might allow 
 to slightly improve our results,
but the present formulation is sufficient for our purpose.
\end{remark}

Regularity of solutions to equations of this kind
has been well studied, pioneered by the technique of
Young (\cite{Youn:36}) and more recently using the theory of rough paths (see 
for example \cite{Lyon:Zhon:02,Friz}
and the references therein). Using fractional derivatives,
 it was shown in \cite[Thm~3]{Hu:Nual:07} that
\begin{equs} 
\sup_{0\leq t\leq T}|\y_t|&\leq 2^{1 + M \|\noise\|^{1/\gamma}_{0,T,\gamma}} |\y_0|\;, \label{eqn:hunu2}
\end{equs}
where  $\|\noise\|_{0,T,\gamma}$ is the H\"older norm defined in
\eref{eqn:Holdnorm} and $M$
is a constant depending on the supremum norms of $f_0$, $f$, $g$ and their first derivatives.
As mentioned earlier, these estimates are not sufficient to obtain moment bounds on the Jacobian $J_{0,t}$ 
and its inverse, since we wish to  consider situations where $f_0$ is unbounded, so that
its supremum norm is not finite.

However it turns out that  an estimate
of the type \eref{eqn:hunu2} holds even if only the derivative of $f_0$ is bounded, thanks to the fact 
that the corresponding ``driving noise'' ($t$) is actually a differentiable function.   
This is the content of the following result:

\begin{lemma} \label{lem:nuexthol} For the processes $x_t$ and $\y_t$ defined in \eref{e:generalequ}, we have
the pathwise bounds:\minilab{e:nuhubounds}
\begin{equs}
\|x\|_{0,T,\gamma} &\leq M (1+ |x_0| )\, \big(1 + \|\noise\|_{0,T,\gamma} \big)^{1/\gamma}\label{eqn:nuhulkbd1}\;,\\
%\|\y\|_{0,T,\infty} &\leq M (1 + |\y_0|) \,2^ {M(1 + \|\noise\|_{0,T,\gamma})^{1/\gamma}}  \label{eqn:nuhulkbd2}\\
\|\y\|_{0,T,\gamma} &\leq M(1+ |x_0| )\, |\y_0| \,e^{M \|\noise\|_{0,T,\gamma}^{1/\gamma}}\;, \label{eqn:nuhulkbd3}
\end{equs}
where $M$ is an absolute constant which depends only on $\|f\|_{\infty}$, $\|f'\|_{\infty}$, $\|f_0'\|_{\infty}$,
 $\|g\|_{\infty}$, $\|g'\|_{\infty}$ and $T$.
\end{lemma}
\begin{proof}
The proof is almost identical to that of \cite{Hu:Nual:07}, so we only try to highlight the main differences. 
For $s,t  \in [0,T]$,
we have the apriori bound (see \cite[p. 403]{Hu:Nual:07}),
\begin{equs}
\left|\int_{s}^t f(x_r)\, d\noise_r \right| \leq M \|\noise\|_{0,T,\gamma} \Big((t-s)^{\gamma} + 
\|x\|_{s,t,\gamma} (t-s)^{2\gamma}\Big)\;,
\end{equs}
where the constant $M$ is indepedent of $\|\noise\|_{0,T,\gamma}$, but depends on $f$.
Since $f_0$ grows at most linearly, we also have
\begin{equ}
\left|\int_{s}^t f_0(x_r)\, dr \right| \leq M  \Big((t -s) + |x_s| (t-s) + \|x\|_{s,t,\gamma} (t-s)^{\gamma +1}\Big) \;,
\end{equ}
and therefore
\begin{equ}
\sup_{s \leq m,n \leq t} 
\left|\int_{m}^n f_0(x_r)\, dr \right| \leq M  \Big((t -s) + |x_s| (t-s) + 2\|x\|_{s,t,\gamma} (t-s)^{\gamma +1}\Big)\;.
\label{eqn:keyfnuestlin} 
\end{equ}
Thus we have 
\begin{equ}
\|x\|_{s,t,\gamma} \leq M_1\Big(1 +\|\noise\|_{0,T,\gamma}  + 
|x_s| (t-s)^{1- \gamma} +(\|\noise\|_{0,T,\gamma} + 1) \|x\|_{s,t,\gamma} (t-s)^{\gamma}\Big)
\end{equ}
for a constant $M_1$ depending on $f$ and the terminal time $T$. 
Set $\Delta = \big (2M_1(\|\noise\|_{0,T,\gamma} + 1)\big)^{-1/\gamma}$.
Then for $|t -s| < \Delta$,
\begin{equs}[eqn:tempholdbd]
\|x\|_{s,t,\gamma} &\leq 2M_1 \Big(1 +\|\noise\|_{0,T,\gamma} + 
|x_s| \Delta^{1- \gamma} \Big)\;.
\end{equs}
Since by definition we have:
\begin{equs}
\|x\|_{s,t,\infty} &\leq |x_s| + \|x\|_{s,t,\gamma} |t-s|^{\gamma}\;,
\end{equs}
for $|t -s| < \Delta$, from \eref{eqn:tempholdbd} it follows that 
\begin{equs}[eqn:delta2m1est]
\|x\|_{s,t,\infty} &\leq |x_s| (1 + 2M_1 \Delta ) + 1\;.
 %2M_1 \Big(1 +\|\noise\|_{0,T,\gamma}  \Big) \Delta^\gamma
\end{equs}
Iterating the above estimate for $N = T/\Delta$, it follows that
\begin{equs}
\|x\|_{0,T,\infty} &\leq |x_0|(1 + 2M_1 \Delta)^{N} + \sum_{k=1}^{N-1}(1 + 2M_1 \Delta )^{k} 
\leq (1 + |x_0|)\,(1 + 2M_1 \Delta)^{N} N \\
&= (1 + |x_0|)\,(1 + 2M_1 \Delta)^{T/\Delta} T/\Delta  \;.
\end{equs}
Using the fact that $(1 + {x \over \Delta})^{\Delta} \leq e^x$ for $\Delta > 0$, this finally yields
\begin{equ}[eqn:xlinfbd]
\|x\|_{0,T,\infty} \leq  (1+ |x_0| ) \,e^{2M_1T}\, T\,(2M_1(\|\noise\|_{0,T,\gamma} + 1)\big)^{1/\gamma}\;.
\end{equ}

%\begin{remark} \label{rem:nuhubd1}
Note that in the last step of the above argument we bounded  $(1 + 2M_1 \Delta)^{T/\Delta}$
from above by $e^{2M_1 T}$ and thus obtained a constant which is exponential in $T$, but \emph{independent}
of $\Delta = \mathcal{O}(\|\noise\|_{0,T,\gamma}^{-1/\gamma})$. If the noise corresponding to the
vector field $f_0$ were \emph{not} differentiable but only H\"older continuous,  then
 instead of  \eref{eqn:delta2m1est},  the exponent would include a power of the H\"older constant of the driving noise,
thus yielding an estimate comparable to those obtained in \cite{Hu:Nual:07}.
%\end{remark}

Thus substituting the bound \eref{eqn:xlinfbd} in \eref{eqn:tempholdbd} yields that,
for $t - s < \Delta$,
\begin{equs} \label{eqn:holdbd01}
\|x\|_{s,t,\gamma} &\leq M (1 + |x_0| ) \big(1 +\|\noise\|_{0,T,\gamma} \big)\;.
\end{equs}
Using the fact the function $f(x) = x^\gamma$ is concave for $\gamma < 1$, it is shown in Lemma \ref{lem:holdgamadd} (proof given in the Appendix) that 
 for  $0= u_0 < u_1<u_2< \cdots< u_{N-1} < u_N=T$,  we have
the bound
\begin{equs}\label{eqn:holdbd021}
\|x\|_{0,T,\gamma} \leq  N^{1- \gamma} \, \max_{0 \leq i \leq N-1} \|x\|_{u_i,u_{i+1},\gamma}  \;.
\end{equs}
Thus from \eref{eqn:holdbd01} and \eref{eqn:holdbd021}
we deduce that
\begin{equs}
\|x\|_{0,T,\gamma} &\leq M (1 + |x_0| ) \big(1 +\|\noise\|_{0,T,\gamma} \big) N^{1- \gamma} =M (1 + |x_0| ) \big(1 +\|\noise\|_{0,T,\gamma} \big)^{1/\gamma}
\;,
\end{equs}
proving the claim made in \eref{eqn:nuhulkbd1}.

Now we come to the second part of Lemma \ref{lem:nuexthol}. Once again, for $s,t \in [0,T]$
we have the apriori bound (see \cite[ p. 406]{Hu:Nual:07})
\begin{equs}
\left|\int_{s}^t g(x_r) \y_r d\noise_r \right| &\leq  \tilde{M}(1 + \|\noise\|_{0,T,\gamma}) \Big(\|g\|_{\infty} \|\y\|_{s,t,\infty}
(t-s)^\gamma \label{eqn:apgbd} \\
&+ (\|g\|_{\infty} \|\y\|_{s,t,\gamma} + \|g'\|_{\infty} \|\y\|_{s,t,\infty} \|x\|_{s,t,\gamma})(t-s)^{2\gamma}\Big)
\end{equs}
for a constant $\tilde{M}$ independent of $\noise$. Thus, setting as before
 $\Delta = \big (2M_1(\|\noise\|_{0,T,\gamma} + 1)\big)^{-1/\gamma}$, we deduce from \eref{eqn:tempholdbd}
and \eref{eqn:apgbd} that,
\begin{equ}
\|\y\|_{s,t,\gamma} \leq  M_2 (1 +\|\noise\|_{0,T,\gamma}) \Big( \|\y\|_{s,t,\infty} 
+  (1 + \|\noise\|_{0,T,\gamma}) \|\y\|_{s,t,\infty}\Delta^{\gamma} +
 \|\y\|_{s,t,\gamma}\Delta^{\gamma}\Big)\;, \label{eqn:estNualcont}
\end{equ}
for a large enough constant  $M_2 \geq \tilde{M}(1 + |x_0| T e^{M_1 T} + \|g\|_{\infty} + \|g'\|_{\infty})$.
From the estimate \eref{eqn:estNualcont}, a straightforward calculation (for instance, see \cite[p. 407]{Hu:Nual:07}) yields
\begin{equs}
\|\y\|_{0,T,\infty} &\leq M 2^ {MT(1 + \|\noise\|_{0,T,\gamma})^{1/\gamma}}|\y_0| \;,\\
\|\y\|_{0,T,\gamma} &\leq M (1+ \|\noise\|_{0,T,\gamma})^{ 2/\gamma}2^{MT{(1 + \|\noise\|_{0,T,\gamma})^{1/\gamma}}}|\y_0| \;,
\end{equs}
for a large constant $M$, thus concluding the proof.
\end{proof}

In the next sections, we will also require bounds on the \textit{second} derivative of the solution flow with respect to
its initial condition. Given an equation of the type \eref{e:generalequ}, the second variation $\z_t \in \R^p$ of $x$ with respect to its initial condition
satisfies an equation of the type
\begin{equ}[e:zhat]
\z_t = \z_0 + \int_0^t h(x_r)(\z_r,\,d\noise_r) + \int_0^t \hat h(x_r)(\y_r,\,\y_r,\,d\noise_r)\;.
\end{equ}
Here, the maps $h\colon \R^n \to \CL_2(\R^p, \R^d; \R^p)$ and $\hat h\colon \R^n \to \CL_3(\R^m,\R^m,\R^d; \R^p)$ are
bounded with bounded derivatives.
Bounds on $\z$ are covered by the following corollary to Lemma~\ref{lem:nuexthol}:

\begin{corollary}\label{cor:secondvariation}
For the process $\z_t$ defined in \eref{e:zhat}, we have
the pathwise bounds:
\minilab{e:nuhubounds}
\begin{equ}
\|\z\|_{0,T,\gamma} \leq M \bigl((1+ |x_0|^5 )(1+ |\y_0|^2 ) + |\z_0|(1+|x_0|)\bigr)\, \exp\big(M\|\noise\|_{0,T,\gamma}^{1/\gamma} \big)\;,
\end{equ}
where $M$ is an absolute constant which depends only on $\|f\|_{\infty}$, $\|f'\|_{\infty}$, $\|f_0'\|_{\infty}$,
 $\|g\|_{\infty}$, $\|g'\|_{\infty}$ and $T$.
\end{corollary}

\begin{proof}
Denote by $\hz_t$ the $\R^{p\cdot p}$-valued solution to the homogeneous equation
\begin{equ}
\hz_t\xi = \xi + \int_0^t h(x_r)(\hz_r\xi,\,d\noise_r)\;,\quad \xi \in \R^p\;.
\end{equ}
Note that the inverse matrix $\hz_t^{-1}$ then satisfies a similar equation, namely
\begin{equ}
\hz_t^{-1}\xi = \xi - \int_0^t \hz_r^{-1} h(x_r)(\xi,\,d\noise_r)\;,\quad \xi \in \R^p\;.
\end{equ}
It follows from Lemma~\ref{lem:nuexthol} that both $\hz$ and $\hz^{-1}$ satisfy the bound \eref{eqn:nuhulkbd3}, just like $\y$ did.
On the other hand, $\z$ can be solved by the variation of constants formula:
\begin{equ}[e:solz]
\z_t = \hz_t \z_0 +  \hz_t \int_0^t \hz_s^{-1} \hat h(x_s)(\y_s,\,\y_s,\,d\noise_s)\;.
\end{equ}
We now use the fact that if $f$ and $g$ are two functions that are $\gamma$-H\"older continuous for some
$\gamma > {1\over 2}$, then there exists a constant $M$ such that 
\begin{equ}[e:boundprod]
\Bigl\|\int_0^\cdot f(t)\,dg(t)\Bigr\|_{0,T,\gamma} \le M ( \|f\|_{0,T,\gamma} + |f(0)|) \,\|g\|_{0,T,\gamma}\;.
\end{equ} 
(See \cite{Youn:36} or for example \cite[Eqn~2.1]{MazNour:08} for a formulation
that is closer to \eref{e:boundprod}.) Furthermore, %there exists a constant $M$ (depending on $T$) such that 
\begin{equ}
\|fg\|_{0,T,\gamma} \le M\bigl(|f(0)| + \|f\|_{0,T,\gamma}\bigr)\bigl(|g(0)| + \|g\|_{0,T,\gamma}\bigr)\;,
\end{equ}
so that, from \eref{e:solz},
\begin{equs}
\|\z\|_{0,T,\gamma} &\le |z_0| \|Z\|_{0,T,\gamma} + C \bigl(1 +  \|Z\|_{0,T,\gamma}\bigr)\bigl(1 +  \|Z^{-1}\|_{0,T,\gamma}\bigr) \\
&\qquad \times \|\hat h'\|_{\infty} \bigl(|x_0| +  \|x\|_{0,T,\gamma} \bigr)\bigl(|y_0| +  \|y\|_{0,T,\gamma} \bigr)^2 \|B\|_{0,T,\gamma}\;.
\end{equs}
The claim now follows by a simple application of Lemma~\ref{lem:nuexthol}.
\end{proof}

The above two lemmas immediately give us the required \emph{conditional} moment estimates for the Jacobian $J_{0,T}$.
\begin{lemma} \label{lem:Jmom} Let $X_t, J_{0,t},m_t$
be as defined in \eref{eqn:sdedem}, \eref{eqn:sflow} and \eref{e:defBtilde}
respectively. Then
for  any $T >0$, $p \geq 1$ and $\gamma \in (1/2, H)$ we have
\begin{equs}
\tE\bigl( \|J_{0,\cdot}\|_{0,T,\gamma}^p \bigr)&< M e^{{ M (1 + \|m\|_{0,T,\gamma})}^{1/\gamma}}\;, \\
~\tE \bigl( \|J^{-1}_{0,\cdot}\|_{0,T,\gamma}^p  \bigr) &< M e^{{ M (1 + \|m\|_{0,T,\gamma})}^{1/\gamma}}\;.
\end{equs}
(Recall that $\tE$ was defined in Remark~\ref{rem:defEtilde}.)
\end{lemma}
\begin{proof}
We first write $B_t = \tilde{B}_t + m_t$ as in \eref{e:defBtilde}. Thus \eref{eqn:sflow}
can be written as:
\begin{equs}
dJ_{0,t} &= V'_0(X_t)J_{0,t} dt + V'(X_t)J_{0,t} \circ (d\tilde{B}_t + dm_t)\;.
\end{equs}
By applying Lemma \ref{lem:nuexthol} to the SDE in \eref{eqn:sflow} with
with $y_t = (t, \tilde{B}_t + m_t)$, $f_0 = V_0$, $f = (0,V_1,\ldots,V_d)$, $g = (V_0',\ldots,V_d')$ 
and using the fact that $1/\gamma >1$ we obtain:
\begin{equs}[eqn:Jpbd]
\|J\|^p_{0,T,\gamma}&
\leq M \,e^{ M (1 + \|\tilde{B}\|_{0,T,\gamma}^{1/\gamma}+ \|m\|_{0,T,\gamma}^{1/\gamma})}\;. 
\end{equs}
The fact that $\gamma > {1\over 2}$ is crucial.
Since $\tilde{B}$ is a Gaussian process, it follows by Fernique's theorem \cite{Bogachev} that
the H\"older norm of $\tilde{B}$ has Gaussian tails, thus proving our first claim.
An identical argument for $J^{-1}_{0,t}$ finishes the proof.
\end{proof}

\begin{remark}
In an identical way, one can obtain similar bounds on the conditional moments of the solution
and of its second variation with respect to its initial condition.
\end{remark}

\section{Smoothness of conditional densities}
\label{sec:smooth}

In this section, we consider again our main object of study, namely the SDE
\begin{equ}
dX_t = V_0(X_t)dt + 
  \sum_{i=1}^dV_i(X_t) \circ d\bth^i \;,\qquad
X_0 = x_0 \in \mathbb{R}^n\;,
\end{equ}
where $B_t$ is a \emph{two-sided} fBM. Our aim is to show that, if the vector fields $V$
satisfy  H\"ormander's celebrated Lie bracket condition (see below), then the \emph{conditional}
distribution of $X_t$ given $\{B_s\,:\, s \le 0\}$ almost surely
has a smooth density with respect to Lebesgue measure on $\R^n$ for every $t>0$. 

In  \cite{Baud:Hair:07}, the authors studied the same equation and showed that the law of $X_t$ 
has a smooth density, but with the key difference that the driving noise $B_t$ was a 
one-sided fBm and no conditioning was involved. Being able to treat the conditioned process will be crucial
for the proof of the strong Feller property later on, this is why we consider it here.
As explained in \eref{e:defBtilde}, we can rewrite $B_t$ as 
$B_t = \tilde{B}_t + m_t$ and thus obtain
\begin{equs}[eqn:modsde]
dX_t  = V_0(X_t)\, dt + V(X_t)\,dm_t + V(X_t) \circ d\tilde{B}_t\;.
\end{equs}
While the distribution of the conditioned process $\tilde{B}_t$ is different from that of
 the one-sided fBm, its covariance satisfies very similar bounds so that the results from \cite{Baud:Hair:07}
 would apply. (See Proposition~\ref{lem:Norr} below.)
 
 However, the difficulty comes from the time derivative
 of the conditional mean $m_t$, which diverges at $t=0$ like $\mathcal{O}(t^{H-1})$.
To use the results from \cite{Baud:Hair:07} directly,  we would need the boundedness of the driving 
vector fields, uniformly in time, so that they are not applicable to our case.
Our treatment of this singularity is inspired by \cite{Hair:Matt:09},
 where a related situation appears due to the infinite-dimensionality of the system considered.
 Using an iterative argument, we will show the invertibility of a slightly modified reduced Malliavin matrix which
 will then, by standard arguments, 
 yield the smoothness of laws of the SDE given in \eref{eqn:modsde}.
 
Before we state H\"ormander's condition, let us introduce a shorthand notation.
For any multiindex $I =(i_1,i_2,\cdots,i_k) \in 
\{0,1,2,\cdots,d\}^k$, we define the commutator:
\begin{equ}
V_I  \equiv [V_{i_1}, [V_{i_2},\cdots,[V_{i_{k-1}},V_{i_k}]\cdots]]\;.
\end{equ}
We also define recursively the families of  vector fields
\begin{equ}
\mathcal{V}_n = \big\{V_{I}, I \in  \{0,\cdots, d\}^{n-1}\times \{1,\cdots,d\}\big \}\;, \label{eqn:setLie}
\qquad \bar{\mathcal{V}}_n = \bigcup_{k=1}^n \mathcal{V}_k\;. \label{eqn:unionLie}
\end{equ}
With these notations at hand, we now formulate H\"ormander's bracket condition \cite{Horm:67}:

\begin{assumption}\label{ass:Hormander}
For every $x_0 \in \mathbb{R}^n$, there exists $N \in \mathbb{N}$
such that the identity
\begin{equ}
\mbox{span}\{U(x_0)\,:\, U \in \bar{\mathcal{V}}_N\} = \mathbb{R}^n
\label{eqn:Horm}
\end{equ}
holds.
\end{assumption}

\subsection{Invertibility of the Malliavin Matrix}

It is well-known from the works of Malliavin, Bismut, Kusuoka, Stroock and others \cite{Mal76,Bismut81,KSAMI,KSAMII,KSAMIII,Norr:86,Mal97SA,Nual:06}
that one way of proving the smoothness of the law
of $X_T$ (at least when driven by Brownian motion) is to first show
the invertibility of the `reduced Malliavin matrix' \footnote{This is a slight misnomer since our SDE is driven by fractional Brownian motion,
rather than Brownian motion. One can actually rewrite the solution as a function of an underlying Brownian motion by making use of 
the representation \ref{e:defBtilde}, but the associated Malliavin matrix has a slightly more complicated relation to $C_T$ than usual.
This will be done in Theorem~\ref{theo:smooth} below.}:
\begin{equ}
\hat{C}_T \eqdef \A_T \A_T^* = \int_{0}^T J_{0,s}^{-1} V(X_s)\,V(X_s)^*(J_{0,s}^{-1})^*~ds\;,
\label{eqn:mallmatorg}
\end{equ}
where we defined the (random) operator: $\A_T : L^2([0,T],\R^d) \mapsto \mathbb{R}^n$ by
\begin{equ}
\A_T v  = \int_{0}^{T} J_{0,s}^{-1}\,V(X_s) \,v(s) \,ds\; .\label{eqn:opA}
\end{equ}
Unfortunately, for the stochastic control argument used in Section~\ref{sec:SF} below for the proof of the
strong Feller property, the invertibility of $\hat{C}_T$ turns out not to be sufficient.
We now define a family of smooth functions $h_T \colon [0,T] \to [0,1]$ :
\begin{equ}[eqn:hfun]
h_T(s) = 
\left\{\begin{array}{cl}
	1 & \text{if $s \le T/2$,} \\
	0 & \text{if $s \ge 3T/4$.}
\end{array}\right.
\end{equ}
and  consider the (random) matrix $C_T$ given by
\begin{equ}[e:modifMall]
C_T \eqdef  \A_T \,h_T \A_T^*\;.
\end{equ}
(See Remark \ref{rem:hintro} below for the reason behind introducing the function $h_T$ and the matrix $C_T$.)
Obviously, $C_T \le \hat{C}_T$ as positive definite matrices, so that invertibility of $C_T$
will in particular imply inertibility of $\hat{C}_T$, but the converse is of course not true in general. Define the matrix norm:
\begin{equs}
\| C_T\| = \sup_{ |\phi| = 1} \langle v, C_T v \rangle, \quad \phi \in \R^n \;.
\end{equs}
Since $C_T$ is a symmetric matrix, by the spectral theorem we have
\begin{equs}
\| C_T^{-1}\| = {1 \over \inf_{ |\phi| = 1} \langle v, C_T v \rangle }, \quad \phi \in \R^n \;.
\end{equs}
The aim of this section is to show that $C_T$ is invertible, and to obtain moment estimates for $\|C_T^{-1}\|$.
 The main result of this
section can be formulated in the following way:

\begin{theorem}\label{thm:mallmatest} 
Assume that the vector fields $\{V_i\}$ have bounded derivatives of all orders for $i \ge 0$ and are bounded
for $i > 0$. Assume furthermore that Assumption~\ref{ass:Hormander} holds, fix $T>0$ and let $C_T$ be as in \eref{e:modifMall}.
Then, for any $R>0$ and any $p \geq 1$, there exists a constant $M$ such that the bound
\begin{equs}
\tprob \Bigl(\inf_{\|\phi\| =1}\langle \phi, C_T \phi\rangle \leq \eps \Big)
\leq M \eps^p\;,\qquad \phi \in \mathbb{R}^n\;,
\end{equs}
holds for all $\eps \in (0,1]$ and for all $x_0$ and $\omega$ such that $|x_0| + \|\omega\|_{(\gamma,\delta)} \le R$.
Here,  $x_0$ is the initial condition of the SDE \eref{eqn:modsde}. 
\end{theorem}

\begin{corollary}
The matrix $C_T$ is almost surely invertible and, for any $R>0$
and any $p \geq 1$, there exists a constant $M$ such that 
 $$\tE  \|C_T^{-p}\| < M\;,$$
uniformly over all $x_0$ and $\omega$ such that $|x_0| + \| \omega\|_{(\gamma,\delta)} \le R$.
\end{corollary}

\begin{proof}
Notice that for $p \geq 1$, we have $\|C_T^{-p}\| = \|C_T^{-1}\|^p$ since $C_T$ is symmetric an positive definite.
Thus
\begin{equs}
\tE \|C_T^{-p}\| &= p\int_{0}^{\infty} x^{p-1} \P(\|C^{-1}_T\| \geq x) \,dx \\
&= p\int_{0}^{\infty} x^{p-1} \tprob \Big(\inf_{\|\phi\|=1} \langle \phi, C^{-1}_T \phi \rangle \leq { 1 \over x} \Big) \,dx < \infty\;,
\end{equs} 
where the last inequality follows from Theorem \ref{thm:mallmatest}, since  for $x > 1$ there
exists a constant $M$ (depending only on $R$) such that 
$\tprob \big(\inf_{\|\phi\|=1} \langle \phi, C^{-1}_T \phi \rangle \leq {1\over x} \big) \leq M x^{-1-p}$
and we are done.
\end{proof}

Before proceeding further, we need the following version of Norris' lemma
\cite{KSAMII,Norr:86}
adapted from \cite{Baud:Hair:07} which will be the key technical
estimate needed for showing the invertibility of $C_T$. In the case of white driving noise, Norris' lemma
essentially states that if a semimartingale is small and if one has some 
\textit{a priori} bounds on its components, then both the bounded
variation part and the martingale part have to be small.
In this sense, it can be viewed as a quantitative version of the Doob-Meyer semimartingale 
decomposition theorem. One version of this result for fractional driving noise is given by:
% For $H \in (1/2,1)$, define the norm:
%\begin{equs}
%\langle \Phi, \Psi \rangle _{H} = \int_{0}^T \int_{0}^T |t-s|^{2H-2} \langle \Phi, \Psi \rangle _{\R^n} \,ds dt,
%~ \Phi, \Psi \in \R^n, T > 0.
%\end{equs}
\begin{proposition}\label{lem:Norr}
Let $H \in ({1\over 2}, 1)$ and $\gamma \in ({1\over 2}, H)$ and let $\tilde B$ be as in \eref{e:defBtilde}. 
Then there exist exponents $q>0$ and $\eta > 0$ such that the following statement holds.

Consider any two families $a^\eps$, $b^\eps$ of processes taking values in
 $\mathbb{R}$ and $\R^{d}$ respectively such that 
$\mathbb{E}(\|a^\eps\|^p_{0,T, \gamma} + \|b^\eps\|^p_{0,T,\gamma}) < 
M_p\eps^{-\eta p}$, uniformly in $\eps \in (0,1]$,
 for some $M_p >0$ and for every $p \geq 1$. Let $y^\eps$ be the process defined by
\begin{equ}
y_t^\eps = \int_{0}^t a_s^\eps\, ds + \int_{0}^t  b_s^\eps \circ d\tilde{B}_s \;,\qquad t \in [0,T]\;.
\end{equ}
Then, the estimate
 \begin{equ}
\mathbb{P}\Big(\|y\|_{0,T,\infty} < \eps ~\mbox{and}  ~\|a\|_{0,T,\infty} + \|b\|_{0,T,\infty}
> \eps^q\Big) < C_p \, \eps^p
\end{equ}
holds for every $p > 0$, uniformly for $\eps \in (0,1]$.
\end{proposition}

\begin{proof} 
Setting $f(s,t) = \E |\tilde B(t) - \tilde B(s)|^2$ as in \cite{Baud:Hair:07}, we have
\begin{equ}
f(s,t) = {t^{2H}\over 2H-1} + {s^{2H}\over 2H-1} - 2 \int_0^s (t-r)^{H-{1\over 2}}(s-r)^{H-{1\over 2}}\,dr\;.
\end{equ}
(We use the convention $s < t$.) A somewhat lengthy but straightforward calculation shows that
this can be rewritten as
\begin{equ}
f(s,t) = |t-s|^{2H} \Bigl({1\over 2H-1} + \int_0^{s  / |t-s|} \bigl((1+x)^{H-{1\over 2}} - x^{H-{1\over 2}}\bigr)^2\,dx\Bigr)\;.
\end{equ}
Similarly, one has
\begin{equ}
\d_s\d_t f(s,t) = -2 \bigl(H - {\textstyle{1\over 2}}\bigr)^2 |t-s|^{2H-2}\int_0^{s  / |t-s|} (1+x)^{H-{3\over 2}} x^{H-{3\over 2}}\,dx\;.
\end{equ}
Since the integrands of both of the integrals appearing in these expressions decay like $x^{2H-3}$, they are integrable so 
that $\tilde B$ is indeed a Gaussian process of `type $H$', see \cite[Eqn~3.5]{Baud:Hair:07}.

The proof is now a slight modification of that of  \cite[Prop~3.4]{Baud:Hair:07}. 
It follows from  \cite[Equ~3.18]{Baud:Hair:07} and the last displayed equation in the proof
 of \cite[Prop~3.4]{Baud:Hair:07} that there exist exponents $q$ and $\beta$ such that the bound
\begin{equ}
\mathbb{P}\Big(\|y\|_{0,T,\infty} < \eps ~\mbox{and}  ~\|a\|_{0,T,\infty} + \|b\|_{0,T,\infty}
> \eps^q\Big) \le C_p \eps^p + \P \bigl(\|b\|_\gamma^2 + \|a\|_\gamma^2 > \eps^{-\beta}\bigr)
\end{equ}
holds. The claim now follows from a simple application of Chebychev's inequality, provided
that we ensure that $\eta$ is such that $2\eta < \beta$.
\end{proof}
For making further calculations easier, we now introduce some definitions that are inspired by the proofs
in \cite{Hair:Matt:09}.
\begin{definition}
A family of events $A \equiv \{A_\eps\}_{\eps \leq 1}$  of a probability space is
called a ``family of negligible events'' if, for every $p \geq 1$ there exists a constant $C_p$ such
that $\mathbb{P}(A_\eps) \leq C_p \eps^p$ for every $\eps \leq 1$.
\end{definition}
\begin{definition}
Given a statement  $\Phi_\eps$, we will say that ``$\,\Phi_\eps$ holds modulo negligible events'',
if there exists a family $\{A_\eps\}$ of negligible events such that,
for every $\eps \leq 1$, the statement $\Phi_\eps$ holds on the complement 
of $A_\eps$.
\end{definition}
With the above definitions, Theorem \ref{thm:mallmatest} can
be restated as:
\begin{theorem}\label{thm:boundC}
Fix $R>0$. Then, modulo negligible events, the bound
$$\inf_{\|\phi\| = 1}\langle \phi, C_T\phi\rangle \ge \eps$$
holds for all $x_0$ and $\omega$ such that $|x_0| + \|\omega\|_{(\gamma,\delta)} \le R$.
\end{theorem}

\begin{remark}
The family of events in question can (and will in general) depend on $R$, but once $R$ is fixed the bounds
are uniform over $|x_0| + \|\omega\|_{(\gamma,\delta)} \le R$. While this uniformity is not really required
in this section, it will be crucial in Section~\ref{sec:SF}.
\end{remark}

\begin{proof}[of Theorem~\ref{thm:boundC}]
It suffices to show that for any fixed $\phi$ with $\|\phi\| = 1$, the bound $\langle \phi, C_T\phi\rangle \ge \eps $ 
holds modulo a family of negligible events. The conclusion then follows by a standard
covering argument, see for example \cite[p.~127]{Norr:86}.
Furthermore, throughout all of this proof, we will denote by $M$ a constant that can change from one expression
to the next and only depends on $R$
and on the vector fields $V$ defining the problem.

Fix now such a $\phi$ and assume ``by contradiction'' that
$\langle \phi, C_T \phi \rangle < \eps$. 
We will show 
that this assumption leads to
the conclusion that
\begin{equs}
\sup_{U \in \bar{\mathcal{V}}_N} |\langle \phi, U(x_0)\rangle | 
&< \eps^{q'}\;,  \quad\hbox{for some $q'>0$,}
\end{equs}
modulo negligible events. However, if H\"{o}rmander's condition is satisfied,
the conclusion above cannot hold since $\{U(x_0), U \in \bar{\mathcal{V}}_N\}$
span $\mathbb{R}^n$ so that there exists a constant $c$ such that $\sup_{U \in \bar{\mathcal{V}}_N} |\langle \phi, U(x_0)\rangle | \ge c$ almost surely,
thus leading to a contradiction and concluding the proof.
For any smooth vector field $U$,
 let us introduce the process $Z_U(t) =  J_{0,t}^{-1} U(X_t)$. With this notation, we have
  
\begin{lemma} \label{lem:hormit1} There exists $r > 0$ such that, modulo negligible events,
the implication
\begin{equ}
\langle \phi, C_T \phi \rangle  < \eps \quad\Rightarrow\quad
\|\langle \phi, Z_U(\cdot)\rangle \|_{0,T/2,\infty} < \eps^r
\end{equ} 
holds for all $U \in \mathcal{V}_1$.
\end{lemma}
\begin{proof} By the definition of $h$, we have the bound
\begin{equ}
\langle \phi, C_T \phi \rangle = \sum_{i=1}^d\int_{0}^T \langle \phi, J_{0,s}^{-1} V_i(X_s)\rangle ^2 \,h_T(s) \,ds \label{eqn:interpref} \ge
 \sum_{i=1}^d\int_{0}^{T/2} \langle \phi, J_{0,s}^{-1} V_i(X_s)\rangle ^2 ds\;.
\end{equ}
We now wish to turn this into a lower bound of order $\eps^r$ (for some $r>0$)
on the supremum of the integrand. Our main tool for this is the bound
\begin{equ}
\|f\|_{0,T, \infty} \leq 2\max \Big(T^{-1/2}\|f\|_{L^2[0,T]} ,
\,\|f\|_{L^2[0,T]}^{{2\gamma \over 2\gamma +1 }} \,
\|f\|_{0,T,\gamma}^{{1 \over 2\gamma +1}} \Big)\;,
\end{equ}
which holds for every $\gamma$-H\"older continuous function $f : [0, T ] \mapsto \R$. 
The proof of this is given in the appendix as Lemma~\ref{lem:inter}.

By  Lemma \ref{lem:Jmom} and
Chebyshev's  inequality it follows that
 $\|\langle \phi, J^{-1}_{0,\cdot} V(X_\cdot)\rangle\|_{0,T,\gamma} < \eps^{-1/2}$ modulo negligible events. Thus if  $\langle \phi, C_T \phi \rangle < \eps$, from Lemma \ref{lem:inter} and
\eref{eqn:interpref} we deduce that, for 
each $1\leq j \leq d$ 
\begin{equ}
\|\langle \phi, J_{0,\cdot}^{-1} V_j(X_\cdot)\rangle \|_{0,T/2,\infty} < M\eps^{{4\gamma -1 \over 2( 2\gamma+1)}}\;,
\label{eqn:interpest}
\end{equ}
modulo negligible events. The claim follows by choosing $r < {4\gamma -1 \over 2( 2\gamma+1)}$.
\end{proof}
We will now use an induction argument to show that a similar bound holds with $V_j$ replaced by any
vector field obtained by taking finitely many Lie brackets between the $V_j$'s. Whilst the
gist of the argument is standard, see \cite{KSAMII,Norr:86,Baud:Hair:07}, 
a technical difficulty arises from the divergence of the derivative of $m(t)$ 
at $t = 0$.  To surmount this difficulty, we take inspiration from \cite{Hair:Matt:09} and introduce a (small) parameter
 $\alpha >0$ to be specified at the end of the iteration. Then, by Lemma \ref{lem:divct}, we have
\begin{equs}[eqn:mcest]
\Bigl\|{dm_{\cdot} \over dt}\Bigr\|_{\eps^{\alpha},T/2,\infty} \leq M\eps^{\alpha(H-1)}\;.
\end{equs}
Instead of considering supremums over the time
 interval $[0,T/2]$, we will now consider supremums over $[\eps^\alpha, T/2]$ instead, thus avoiding the singularity at zero.
 
In order to formalise this, for $r>0$ and $i \ge 1$, define the events
\begin{equ}
\K_i^\eps(\alpha,r) = \bigl\{\|\langle \phi, Z_U(\cdot)\rangle \|_{\eps^\alpha,T/2,\infty} < \eps^{r} \;,\quad \forall\; U \in \mathcal{V}_i\bigr\}\;.
\end{equ}
With this notation at hand, the recursion step in our argument can be formulated as follows:
 
\begin{lemma}\label{lem:hormit2}
There exists $\bar q>0$ such that, for every $\alpha > 0$ and every $i \ge 1$, the implication
\begin{equ}
\K_i^\eps(\alpha,r) \quad\Rightarrow\quad \K_{i+1}^\eps\bigl(\alpha,\bar qr + \alpha(H-1)\bigr)\;,
\end{equ}
holds modulo negligible events. %Here, $q$ is as in Proposition~\ref{lem:Norr}.
\end{lemma}
\begin{proof}
Assume that $\K_i^\eps(\alpha,r)$ holds. For any $t \leq T/2$, it then follows from the chain rule that
\begin{equs}
\langle \phi, &Z_U(t) -  Z_U({\eps}^{\alpha}) \rangle = 
\int_{\eps^{\alpha}}^t \langle \phi, J_{0,s}^{-1} [V_0,U](X_s)\rangle \,ds  \label{eqn:Liebrackcal}  \\
& +\sum_{i=1}^d \int_{{\eps}^\alpha}^t \langle \phi, J_{0,s}^{-1} [V_i,U](X_s)\rangle \,dm^{i}_{s}
+\sum_{i=1}^d\int_{\eps^{\alpha}}^t \langle \phi, J_{0,s}^{-1} [V_i,U](X_s)\rangle
\circ d\tilde{B}_s^i\;,
\end{equs}
where $\tilde{B}_{\cdot}$ is as in \eref{e:defBtilde}.

Now by the hypothesis,
Proposition~\ref{lem:Norr} and  \eref{eqn:Liebrackcal}, it follows that
for any $U \in \mathcal{V}_i$,
\minilab{eqn:norrit}
\begin{equs}
\Big\|\Big\langle \phi,  \sum_{i=1}^d 
J_{0,\cdot}^{-1} [V_i,U](X_\cdot) \Big\rangle \Big\|_{\eps^\alpha,T/2,\infty} &< M\eps^{rq}\;, \label{eqn:norriti} \\
%\label{eqn:norrit1} \\
\Big\|\Big\langle \phi, J_{0,\cdot}^{-1} [V_0,U](X_\cdot)  +  \sum_{i=1}^d {dm^{i}_{\cdot} \over ds}
J_{0,\cdot}^{-1} [V_i,U](X_\cdot) \Big\rangle \Big\|_{\eps^\alpha,T/2,\infty} &< M\eps^{rq}\;, \qquad \label{eqn:norrit0}
\end{equs}
modulo negligible events, with $q$ as given by Proposition~\ref{lem:Norr}. Plugging \eref{eqn:mcest} and \eref{eqn:norriti} into 
\eref{eqn:norrit0}, it follows that for every
$U \in \mathcal{V}_i$, 
\begin{equs}
\|\langle \phi, J_{0,\cdot}^{-1} [V_0,U](X_\cdot) \rangle \|_{\eps^\alpha,T/2,\infty}\;
&< M \eps^{rq + \alpha(H - 1)}
\end{equs}
modulo negligible events. Choosing any $\bar q < q$, the claim then follows from the fact that
 for every exponent $\delta > 0$, one has $M \eps^\delta < 1$ modulo negligible events. (Recall that $H<1$, so that \eref{eqn:norriti}
 is actually slightly better than the claimed bound.)
\end{proof}

Now we iterate the previous argument.
Let $N \in \mathbb{N}$ be such that
\begin{equ}[e:boundbelow]
\inf_{|x| \le R} \inf_{|\xi| = 1} \sum_{U \in \bar \CV_N} |\scal{\xi, U(x)}|^2 > 0\;.
\end{equ}
It follows from the H\"ormander condition and the fact that the vector fields are smooth that such an $N$ exists.
From Lemmas \ref{lem:hormit1} $\&$ \ref{lem:hormit2}, we obtain that 
for every
$U \in \bar{\mathcal{V}}_N$, modulo negligible events,
\begin{equs}
\|\langle \phi, J_{0,\cdot}^{-1} U(X_\cdot) \rangle \|_{\eps^\alpha,T/2,\infty}
&<  \eps^{q_1}\;,
\label{eqn:norrisfinalest}
\end{equs}
where $q_1 = r \bar q^N + \alpha (H-1)\frac{(1 - (r\bar q)^N)}{1-r\bar q}$.  Since $\bar q$ and $r$ are fixed,
we can choose $\alpha$ small enough such that $q_1 > 0$. 

Now we make use of the 
H\"older continuity of $J^{-1}_{0,\cdot}$ and $X$. Indeed, 
it follows from Lemma \ref{lem:Jmom} that
\begin{equs}
\|J^{-1}_{0,\cdot}U(X_\cdot)  - U(x_0)\|_{0,\eps^\alpha,\infty} < \eps^{q_2}\;,
\label{eqn:contliebra}
\end{equs}
modulo negligible events, for any $q_2 < \alpha \gamma$. Combining  \eref{eqn:norrisfinalest}
and \eref{eqn:contliebra}, we conclude that
\begin{equs}
\sup_{U \in \bar{\mathcal{V}}_N} |\langle \phi, U(x_0)\rangle | 
&< \eps^{q_{1}} + \eps^{q_{2}}\;,
\end{equs}
modulo negligible events. This is in direct contradiction with \eref{e:boundbelow}, thus concluding the proof of
Theorem~\ref{thm:boundC}.
\end{proof}

We conclude that the matrix $\tilde{C}_T$ defined in 
\eref{eqn:mallmatorg} is almost surely invertible and $\tE(\|\tilde{C}_T\|^{-p} )
< M$
for $p \geq 1$. As a consequence of Theorem \ref{thm:mallmatest}, we obtain
the smoothness of the laws of $X_t$ conditional on a given past $\omega \in \CW$: 
\begin{theorem}\label{theo:smooth}
Let $\phi_t(x,\tilde \omega)$ be the solution to \eref{e:SDEmain} with notations as in Section~\ref{sec:basicdef},
and assume that Assumption~\ref{ass:Hormander} holds. Then, for every $\omega \in \CW$, the
distribution of $\phi_t(x,\tilde \omega)$ under $\CP(\omega,\cdot\,)$ has a smooth density with respect
to Lebesgue measure.
\end{theorem}

\begin{proof}
In order to obtain the smoothness of the densities, it suffices by \cite{Nual:06,Mal97SA,Baud:Hair:07,NuaSau09} to show
that the Malliavin matrix $\CM_T$ of the map $W \mapsto X_T^x$, where $W$ is the standard Wiener process
describing the solution as in \eref{e:defBtilde}, is invertible with negative moments of all orders. 
As we will see in Section~\ref{sec:SF} below (see equation~\ref{eqn:contbasm}), $\CM_T$ is given by
\begin{equ}
\scal{\xi,\CM_T\xi} = \bigl\| \bigl(\CI^{H - {1\over 2}} \bigr)^* \CA^* J_{0,T}^*\xi \bigr\|_{L^2}^2\;,
\end{equ}
where the $L^2$-adjoint of the fractional integral operator $\CI^{H-{1\over 2}}$ is given by
\begin{equ}
\bigl(\bigl(\CI^{H - {1\over 2}} \bigr)^*g\bigr)(t) = \int_t^T (s-t)^{H-{3\over 2}} g(s)\,ds\;.
\end{equ}
We already know from Theorem~\ref{thm:mallmatest} that $\|\CA^* J_{0,T}^*\xi\|_{L^2}^2 \ge \eps$ modulo
negligible events and we would like to make use of this information in order to obtain a lower bound on $\CM_T$.

Since $\CI^{H - {1\over 2}}$ and $\CD^{H - {1\over 2}}$ are inverses of each other,
for any measurable $f$ we have,
\begin{equs}
\|f\|^2_{L^2}  &= \scal{f, \CI^{H - {1\over 2}} \CD^{H - {1\over 2}} f}_{L^{2}}  \le \bigl\|\bigl(\CI^{H - {1\over 2}} \bigr)^*f\bigr\|_{L^{2}} \|\CD^{H - {1\over 2}} f\|_{L^{2}} \\
&\le M\bigl\|\bigl(\CI^{H - {1\over 2}} \bigr)^*f\bigr\|_{L^{2}} \|f\|_{0,T,\gamma}\;.
\end{equs}
Thus we obtain,
\begin{equ}
\bigl\| \bigl(\CI^{H - {1\over 2}} \bigr)^* \CA^* J_{0,T}^*\xi\bigr\|_{L^{2}} \ge M \|\CA^* J_{0,T}^*\xi\|_{L^{2}}^{2} \|\CA^* J^*_{0,T} \xi\|_{0,T,\gamma}^{-1}\;.
\end{equ}
Since we know from the results in Section~\ref{sec:bounds} that $ \|\CA^* J_{0,T} \xi\|_{0,T,\gamma} \le \eps^{-1}$ modulo negligible events,
we conclude that
\begin{equ}
\bigl\| \bigl(\CI^{H - {1\over 2}} \bigr)^* \CA^* J_{0,T}^*\xi\bigr\|_{L^{2}}  \ge \eps \|\CA^* J_{0,T}^*\xi\|_{L^{2}} ^{2}\;,
\end{equ}
holds modulo negligible events. The proof now follows from Theorem~\ref{thm:mallmatest}.
\end{proof}

\section{Strong Feller Property}
\label{sec:SF}

Recall from the arguments in Section~\ref{sec:ergodic} that the strong Feller property \eref{eqn:SF} is the only ingredient
required for the proof of Theorem~\ref{thm:main}. We will actually show something slightly stronger than just \eref{eqn:SF}.
Fix some arbitrary final time $T>1$, a Fr\'echet differentiable bounded map $\psi\colon \CC([1,T],\R^n)\to \R$ with bounded
derivative\footnote{The space $\CC([1,T],\R^n)$ is endowed with the usual supremum norm}, denote by $R_{1,T}\colon \CC(\R_+,\R^n)\to \CC([1,T],\R^n)$ the restriction operator, and set as before
\begin{equs}
\bar \CQ \psi (x,\omega) \eqdef \int_{\CC(\R_+,\R^d)} \psi(R_{1,T} z)\Q(x,\omega; dz)\;.
\end{equs}
With these notations, the main result of this article is:

\begin{theorem}\label{thm:mainSF}
Assume that the vector fields $\{V_i\}_{i \ge 0}$ have bounded derivatives of all orders and that furthermore the vector
fields $\{V_i\}_{i\ge 1}$ are bounded. Then, if H\"ormander's bracket condition (Assumption~\ref{ass:Hormander}) holds, 
there exists a continuous function
$K\colon \R^n \to \R_+$ and a constant $M>0$ such that for $\omega \in \CW_{\gamma,\delta}$ and $x \in \R^n$, the bound
\begin{equs}[eqn:polyexpbd]
|D_x \Q \psi(x,\omega)| \leq K(x) \,e^{M \|\CG \omega\|_{0,T,\gamma}^{1/\gamma}}\;,
\end{equs}
holds uniformly over all test functions $\psi$ as above with $\|\psi\|_\infty \le 1$. Here, $\CG$ was defined in \eref{eqn:opera} and both $K$ and $M$ are independent of $T$ and of $\psi$.
In particular, \eref{e:SDEmain} is strong Feller.
\end{theorem}

The remainder of this section is devoted to the proof of this result.
For conciseness, from now onwards we will use the shorthand notation  
\begin{equs}
|\Psi(x,\omega)| < \tp\;,
\end{equs}
to indicate that there exist a continuous function $K$ on $\R^n$ and a constant $M$ such that the expression $\Psi(x,\omega)$ is bounded
by the right hand side of \eref{eqn:polyexpbd}.

 To obtain the bound for $D \Q \psi(\cdot, \omega)$, we use
 a  `stochastic control' argument which
 is inspired by \cite{Hair:Matt:06,Hair:Matt:09} and is similar in spirit to the arguments
 based on
 the  Bismut-Elworthy-Li formula \cite{Xuemei,BEL}. We will explain this argument in more detail in Section~\ref{sec:control} below,
but  before that let us introduce the elements of Malliavin calculus required for the construction.

Let $X^x_t$ and $J_{0,t}$ be the solution and its Jacobian, as defined in
 equations \eref{eqn:sdedem} and  \eref{eqn:sflow} respectively with initial condition $X_0^x =x$. Let $\phi_t(x,\tilde \omega) 
 = X_t^x$ denote the flow map as in Section~\ref{sec:basicdef}. Henceforth we will use $\phi_t(x,\tilde \omega)$ and $X^x_t$
 interchangeably, with a preference for the first notation when we want to emphasise the dependence of the solution either on the
 realisation of the noise or on its initial condition.

By the variation of constants formula, 
the Fr\'echet derivative of the flow map \textit{with respect to the driving noise} $\tilde \omega$ 
in the direction of $\int_0^\cdot v(s)\,ds$ is then given by
\begin{equs}
D_v X^x_t = D_v\phi_t(x,\tilde\omega) = J_{0,t}\A_t v \;, \label{eqn:contbas}
\end{equs}
where $\A_t$ is as defined in \eqref{eqn:opA}.
\begin{remark}
Notice that the operator $D$  defined above is a standard Fr\'echet derivative unlike in the case
$H = \frac{1}{2}$ where it would have to be interpreted as a Malliavin derivative.
 This is because the  integration with respect to the fBM for $H > \frac{1}{2}$ is just the 
 Reimann-Stieltjes integral, whereas for $H = \frac{1}{2}$ the stochastic integral is not 
 even a continuous (let alone differentiable) map of the noise in the case of multiplicative noise.
\end{remark}
As already mentioned in Remark \ref{rem:idinvers}, the operator $\CD^{H-{1\over 2}}$ and its inverse $\CI^{H-{1\over 2}}$ 
provide an isometry
between the Cameron-Martin space of the Gaussian measure $\CP(\omega,\cdot\,)$ and that of Wiener measure. 

As a consequence, if $v$ is such that $\tilde v = \CD^{H-{1\over 2}} v \in L^2(\R_+,\R^d)$, then we have the identity
 \begin{equ}
\D_{\tilde{v}} X_t^x  = D_v\phi_t(x,\tilde\omega) = J_{0,t}\A_t v = J_{0,t}\A_t\, \CI^{H-{1\over 2}}\tilde v\;,  \label{eqn:contbasm}
 \end{equ}
where $\D$ is the Malliavin derivative of $X_t^x$ when interpreted as a function on Wiener space
via the representation \eref{e:defBtilde}.

Before we continue, let us make a short digression on the definition of the Malliavin derivative that appears in the
above expression.
In general, if $Z \colon \widetilde{\CW} \to \R$ is 
any Fr\'echet differentiable function, then its Malliavin derivative with respect to the underlying Wiener process is
given by
\begin{equ}[e:def:Mal]
\D_{\tilde v} Z \eqdef D_{v}Z\;,
\end{equ}
where $D_vZ$ denotes the Fr\'echet derivative of $Z$ in the direction $\int_0^\cdot v(t)\,dt$ and where $\tilde v$
and $v$ are related by $\tilde v = \CD^{H-{1\over 2}} v$ as before.

Since, for every $T>0$, the map $\CD^{H-{1\over 2}}$ restricted to functions on $[0,T]$
is bounded from $\widetilde\CW$ to $\CC([0,T],\R^d)$ and since the dual of the latter space
consists of finite signed measures, we conclude from \eref{e:def:Mal} that if $Z$ is Fr\'echet differentiable, then there 
exists a function $s \mapsto \D_s Z$ which is locally of bounded variation and such that the 
Malliavin derivative of $Z$ can be represented as
\begin{equ}
\D_{\tilde v} Z = \int_0^\infty \scal{\D_s Z, \tilde v(s)}\,ds\;.
\end{equ} 
(The scalar product is that of $\R^d$ here.)
The quantity $\D_s Z$ should be interpreted as the variation of $Z$ with respect to the infinitesimal increment
of the underlying Wiener process $W$ at time $s$.

With this convention, it follows from \eref{eqn:contbasm} and the definition of $\CI ^{H-{1\over 2}}$ 
that for some constant $c_H$ one has the identity
\begin{equ}[e:MalliavinDer]
\D_s X_t^x = c_H \int_{s}^t J_{r,s} \,V(X_r^x)\, (r -t)^{H-{3\over 2}} \,dr \;,
\end{equ}
for $s \le t$. Here, $V$ denotes the matrix $(V_1,\ldots,V_d)$ as before. 
Note also that $\D_s X_t^x = 0$ for $s > t$ since $X$ is an adapted process. 
Of course, the whole point of Malliavin calculus is to also be able to deal with functions
that are not Fr\'echet differentiable, but for our purposes the framework discussed here will suffice.

\subsection{Stochastic control argument}
\label{sec:control}

At the intuitive level, the central idea of the stochastic control argument is as follows. 
In a nutshell, the strong Feller property states that for two nearby initial conditions $x_0$ and $y_0$
and an arbitrary realisation $\omega$ of the \textit{past} of the driving noise, 
it is possible to construct a coupling between the solutions $x_t$ and $y_t$ such that with high probability
(tending to $1$ as $y \to x$), one has $x_t = y_t$ for $t \ge 1$. One way of achieving such a coupling is 
to perform a change of measure on the driving process for one of the two solutions (say $y$) in such a way,
that the noise $\tilde \omega_y$ driving $y$ is related to the noise $\tilde \omega_x$ driving $x$ by
\begin{equ}
d \tilde \omega_y(t) = d \tilde \omega_x(t) + v_{x,y}(t)\,dt\;,
\end{equ}
where $v$ is a `control' process that aims at steering the solution $y$ towards the solution $x$.
If one takes $y = x + \eps \xi$ and looks for controls of the form $v_{x,y} = \eps v$, then in the limit
$\eps \to 0$, the effect of the control $v$ onto the solution $x$ after time $t$ is given
by \eref{eqn:contbasm}, while the effect of shifting the initial condition by $\xi$ is given by $J_{0,t}\xi$.
It is therefore not surprising that in order to prove the strong Feller property, our main 
task will be to find a control $v$ such that $J_{0,1}\A_1 v = J_{0,1}\xi$, or equivalently
$\A_1 v = \xi$. This is the content of the following proposition:

\begin{proposition}\label{prop:IBP}
Assume that, for every $x \in \R^n$ and $\omega \in \CW$, 
there exists a stochastic processes $v  \in \CC^\gamma(\R_+,\R^d)$ such that the identity
\begin{equ}
\CA v = \xi\;,
\end{equ}
holds almost surely, where $\A \eqdef \A_1$  defined in \eref{eqn:opA}. 
Assume furthermore that the map $\tilde w \mapsto v$ 
is Fr\'echet differentiable from $\widetilde \CW$ to  $\CC^\gamma(\R_+,\R^d)$. Finally, we assume that 
the Malliavin derivative of the process $\tilde v \eqdef \CD^{H-{1\over 2}} v$ 
satisfies the
bounds
\begin{equ}[e:boundv]
M_1 = \tE \int_0^\infty |\tilde v(s)|^2\,ds < \infty\;,\quad 
M_2 = \tE\int_0^\infty \int_0^\infty |\D_t \tilde v(s)|^2\,ds\,dt < \infty\;.
\end{equ}
Then, the bound
\begin{equ}
|\langle D_x \Q \psi (x, \omega), \xi \rangle| \le \|\psi\|_\infty \sqrt{M_1 + M_2}\;,
\end{equ}
holds uniformly over all test functions $\psi$ as in Theorem~\ref{thm:mainSF}, all $T>1$, and all $\xi \in \R^n$.

In particular, if $v$ is such that $M_1 + M_2 < \tp$, then the conclusions of Theorem~\ref{thm:mainSF} hold.
\end{proposition}

\begin{remark}
Since $\CD^{H-{1\over 2}}$ is bounded from $\CC^\gamma(\R_+,\R^d)$ into $\CC(\R_+,\R^d)$, the Malliavin
derivative of $\tilde v$ exists, so that the expression in \eref{e:boundv} makes sense.
\end{remark}

\begin{proof}
 Given an initial displacement $\xi \in \R^n$, 
we seek for a `control' $v$ on the time interval $[0,1]$ that solves the equation $\A v = \xi$, where we use the
shorthand notation $\A = \A_1$. 
If this can be achieved,
then we extend $v$ to all of $\mathbb{R}_+$ by setting $v(s) = 0$ for $s \ge 1$ and we define $\tilde v = \CD^{H-{1\over 2}} v$.
Note that since $v(s) = 0$ for $s > 1$, it follows from the definition of $\CA_t$ that we have $\CA_t v = \CA v = \xi$ for $t \ge 1$.
If $v$ is sufficiently regular in time so that this definition makes sense, we thus have the identity 
\begin{equs}
\D_{\tilde{v}} X^x_T = J_{0,T} \xi\;,
\label{eqn:contMalli}
\end{equs}
for every $T \ge 1$. In particular, for a $\CC^1$ test function $\psi$ as  above and
for  $\xi \in \mathbb{R}^n$, 
\begin{equs} 
\langle D_x \Q \psi (x, \omega), \xi \rangle &= \tE\Big((D \psi)(\phi(x,\tilde\omega)) J_{0,\cdot}\xi\Big)\\
&= \tE\Big((D \psi)(\phi(x,\tilde\omega)) \,\D_{\tilde{v}} X_\cdot\Big) = \tE\big(\D_{\tilde{v}} \bigl(\psi(\phi(x,\tilde\omega))\bigr)\big)\;.
\end{equs}
It then follows from the integration by parts formula of Malliavin calculus \cite{Nual:06} that
this quantity is equal to
\begin{equ}
\ldots =\tE\Big( \psi(\phi_T(x,\tilde\omega)) \int_{0}^T  \scal{\tilde{v}(s), dW_s}  \Big)
\leq \|\psi\|_{\infty}\tE\,\Big|\int_{0}^T \scal{\tilde v(s) , dW_s} \Big|\;.
\end{equ}
The stochastic integral appearing
in this expression is in general not an It\^o integral, but
a Skorohod integral, since its integrand is in general not adapted to the filtration generated by $W$.

For Skorohod integrals one has the following extension of It\^o's isometry
 \cite{Nual:06}:
\begin{equs}
\tE\Big( \int_{0}^{\infty} \scal{\tilde v(s) , dW_s} \Big)^2
&= \tE\Big( \int_{0}^{\infty} |\tilde v(s)|^2 ds \Big)
 + \tE \int_{0}^{\infty} \int_{0}^{\infty}
\tr \bigl(\D_t \tilde v(s)^T\, \D_s\tilde v(t)\bigr)\,ds\, dt\\
&\le M_1 + M_2\;, \label{eqn:integsko}
\end{equs}
where as before $\D$ denotes Malliavin derivative with respect to the underlying Wiener process.
The claim then follows at once.
\end{proof}

 In order to have the identity \eref{eqn:polyexpbd}, it therefore remains to find a control $v$ satisfying $\CA v = \xi$ and
 such that
\begin{equ}
M_1 + M_2 < \tp \;,
\label{eqn:scito}
\end{equ} 
where the quantities $M_1$ and $M_2$ are defined by \eref{e:boundv} in terms of $\tilde v = \CD^{H-{1\over 2}}v$.
A natural candidate  for such a control $v$ can be identified by
a simple least squares formula. Indeed, note that the 
 adjoint $\A^*: \mathbb{R}^n \mapsto L^2(\R_+, \R^d)$ of $\A$ is given by:
\begin{equ}[eqn:astar]
(\A^* \xi)(s)  =  {V(X_s)}^* (J_{0,s}^{-1})^*\xi\;,\quad \xi \in \mathbb{R}^n\;.
\end{equ}
Let $h = h_1$ where $h_1 \colon [0,1] \to \R_+$  is as  defined in \eref{eqn:hfun}. 
Assuming that $\A h \A^*$ is invertible,
one possible solution to the equation $\A v = \xi$ is given by
\begin{equ}
v(s) \equiv h(s)\, \big(\A^* (\A h \A^*)^{-1}\xi \big)(s)\;. \label{eqn:stochcont}
\end{equ}
Note that $\A h \A^*$ is precisely equal to the `modified Malliavin matrix' $C_1$ defined in \eref{e:modifMall}.
By the results from the previous section the matrix $C_1$ is almost surely invertible and therefore
 the expression \eqref{eqn:stochcont} does make sense.
 
 The aim of the next section is to show that the choice \eref{eqn:stochcont} does indeed
 satisfy the assumptions of Proposition~\ref{prop:IBP}.

\begin{remark}\label{rem:hintro}
For the identity \eref{eqn:contbasm} to hold, 
the stochastic control $v(\cdot)$ needs to belong to the 
Cameron-Martin space of $\CP(0,\cdot)$, and
 the function $h(\cdot)$ is introduced 
just for this purpose.
Indeed, since we extend the control $v(\cdot)$ to vanish outside the interval $[0,1]$,
one needs to be careful about the regularity of $v(s)$ at $s =1$.
It can be shown using the fractional integrals  that, in order
for $v(\cdot)$ to have the required regularity, the function
$h(s)$ should be  $\mathcal{O}( (1-s)^{H})$
for $s \approx 1$. The specific $h_1(\cdot)$ we chose in \eref{eqn:hfun},
of course, satisfies this.
\end{remark}

\subsection{Proof of the main result, Theorem~\ref{thm:mainSF}}

In view of Proposition~\ref{prop:IBP}, the proof of Theorem~\ref{thm:mainSF}
is complete if we can show that the stochastic process $v$ defined by 
\eref{eqn:stochcont} satisfies the bounds \eref{eqn:scito}. This will be the content
of Propositions~\ref{prop:Holdertrick} and \ref{prop:sectermthm} below.

We start with an estimate on the H\"older norm of $v$ which will
be used repeatedly. Here and in the rest of this section,
we will make use of the notation
\begin{equ}
\|v\|_{\gamma} \eqdef \|v\|_{0,1,\gamma} + \|v\|_{0,1,\infty}
\end{equ}
for the H\"older norm of $v$. 
This has the advantage of being a norm rather than just a seminorm and,
 for any two H\"older continuous functions $u$ and $v$, we have the bound
\begin{equ}
\|u v\|_{\gamma} \leq M\|u\|_{\gamma} \|v\|_{\gamma}\;,
\end{equ}
for some fixed constant $M>0$. With this notation, our bound is:

\begin{lemma}\label{lem:vcondhold} The stochastic control $v$ defined
in \eref{eqn:stochcont} is continuous and satisfies the bound
\begin{equ}
\tE{\|v\|^2_{\gamma}} < \tp\;.
\end{equ}
\end{lemma}
\begin{proof}
Setting $v(s) = h(s) u(s)$, we obtain from \eref{eqn:stochcont} the bound
\begin{equs}[eqn:vtempes]
\|u\|^2_{\gamma} \leq  \|\A^* \eta\|^2_{\gamma}\;,
\end{equs}
where $\eta = (\A h \A^*)^{-1}\xi$.
From \eref{eqn:astar} and the fact that $V$ has bounded derivatives 
it follows that
\begin{equ}
\|\A^* \eta \|_{\gamma} \leq M \|V(X_\cdot)\|_{\gamma}\|J^{-1}_{0,\cdot}\|_{\gamma} |\eta| \leq 
M \|X_\cdot\|_{\gamma}\|J^{-1}_{0,\cdot}\|_{\gamma} |\eta|\;,
\end{equ}
and therefore by Cauchy-Schwarz
\begin{equ}
\Big(\tE \|\A^* \eta \|^2_{\gamma}\Big)^2  \leq M \tE \big(\|X_\cdot\|^4_{\gamma}\|J^{-1}_{0,\cdot}\|^4_{\gamma}\big)\, \tilde \E |\eta|^4 \;.
\end{equ}
From  Lemma \ref{lem:Jmom} we deduce that $\tE \big(\|X_\cdot\|^4_{\gamma}\|J^{-1}_{0,\cdot}\|^4_{\gamma}\big) < \tp$. 
Also from Theorem \ref{thm:mallmatest}, we have $\tE |(\A h \A^*)^{-1}\xi|^p < \tp$ for any $p \geq 1$. Thus we obtain 
\begin{equs}
\tE \|u\|^2_{\gamma} \leq  \tp\;. 
\end{equs}
The claim then follows immediately from $v = hu$.
\end{proof}

\begin{remark}
Had we not introduced the control $h$, we could still obtain continuity of $v$ on $[0,1]$, but we would 
have a discontinuity at $t=1$.
\end{remark}

As a consequence, we show that $\tilde v = \CD^{H-{1\over 2}}v$ is square integrable:
\begin{proposition} \label{prop:Holdertrick} For the stochastic control $v(\cdot)$ defined
in \eref{eqn:stochcont} and the operator $\mathcal{D}$ defined in \eref{e:frac}, the
following estimate holds:
\begin{equs}
\tE\Big( \int_{0}^{\infty} |\tilde v(s)|^2 ds \Big) < \tp\;.
\end{equs}
\end{proposition}

%\begin{remark}
%One might think that Proposition~\ref{prop:Holdertrick} is sufficient to show \eref{eqn:scito}. This is not the case since
%$\tilde v$ is not adapted to the filtration generated by $W$ and so It\^o's isometry does not hold.
%\end{remark}

Before we turn to the proof of Proposition~\ref{prop:Holdertrick}, we state the following useful result:
\begin{lemma} \label{lem:Dalpop} 
For any $s \geq 0, \,\alpha > 0$ and $f \in \CC^{\gamma}$ with $\gamma> \alpha$ ,
\begin{equs}
\Gamma( 1 - \alpha)\, \CD^{\alpha} f(s) &= \frac{d}{ds} \int_0^{s} (s -r )^{-\alpha} f(r) \,dr \\
&=   s^{-\alpha} f(s) - \alpha \int_{0}^s
(s -r )^{-\alpha -1} \big(f(r) - f(s)\big)\,dr \;.
\label{eqn:Marchder}
\end{equs}
\end{lemma}
\begin{proof} See \cite[Thm~13.1]{Sam:Kil:Mar:93}, as well as
Lemma 13.1 in the same monograph. 
\end{proof}

\begin{remark}
The right hand side of equation \eref{eqn:Marchder} 
is called the Marchaud  fractional 
derivative. Lemma \ref{lem:Dalpop} actually holds true
for a larger class of functions $f$, namely those such that $f = \CI^{\alpha}g$ for
some $g \in L^1([0,1],\R)$.
\end{remark}

\begin{proof}[of Proposition~\ref{prop:Holdertrick}]
The main idea behind this proof, and a recurrent theme 
in this section is the following. Although the stochastic control $v(\cdot)$
vanishes outside $[0,1]$, the integrand  $\tilde v(s)$
in non-zero for all $s \in [0,\infty)$. Thus the square integrability
of $\tilde v(s)$ on the interval $[0,\infty)$ is shown 
by dealing with the singularities of the integrand differently at $s = 0$,  and $s = \infty$.
At $s =0$, we use the fact that $v$ is H\"older continuous whereas  at $s = \infty$,
we use the fact that $v$ is bounded.

By definition,
\begin{equ}
\tilde v(s) = \frac{d}{ds}\int_{0}^{s} \,(s-r)^{\frac12 - H} v(r)\,dr\;.
\end{equ}
Applying Lemma \ref{lem:Dalpop} with $\alpha =H- \frac12$, we obtain the bound
\begin{equ}
|\tilde v(s)| %=\frac{d}{ds}\int_{0}^{1} \,(s-r)^{\frac12 - H}\, v(r)dr
\leq M s^{\frac12 - H} |v(s)| + M \int_{0}^s
(s -r )^{-\frac12 - H } \big| v(r) - v(s)\big| \,dr \;,
%&\leq M s^{\frac12 - H} v(s) + M \|v\|_{\gamma}\int_{0}^s
%(s -r )^{-\frac12 - H + \gamma}dr \\
%&\leq M s^{\frac12 - H} v(s) + M \|v\|_{\gamma}\,s ^{\frac12 - H + \gamma} \\
\end{equ}
for some constant $M>0$.
Noting that the  control 
$v(\cdot)$ vanishes outside the interval $[0,1]$, we have
\begin{equs}
\tE\Big( \int_{0}^{\infty} |\tilde v(s)|^2 ds \Big)
&\leq M \tE \|v\|_{\gamma}^2 \int_{0}^{1} s^{1 -2 H}\,ds \label{eqn:integsmall} \\
&\quad + 
M\tE \int_{0}^{\infty}\Big( \int_{0}^s
(s -r )^{-\frac12 - H } \big| v(r) - v(s)\big| \,dr \Big)^2 ds \;.
\end{equs}
The first  integral 
$\int_{0}^{1} s^{1 -2 H}\,ds$ in equation \eref{eqn:integsmall} is clearly finite since
$H < 1$, so that this term is bounded by $\tp$ by Lemma~\ref{lem:vcondhold}. Now we show that the 
second term has a similar bound. To this end, set
\begin{equ}
\zeta(s) \equiv  \int_{0}^s
(s -r )^{-\frac12 - H } \big| v(r) - v(s)\big| \,dr \;.
\end{equ}
%Since $\zeta(\cdot)$, it is sufficient to study the possible
%singularities of  $\zeta(s)$  at $s =0$ and $s = \infty$. As mentioned 
%earlier we deal with
%these two singularities differently.  
%Thus,
For $s<1$, we obtain
\begin{equs}
|\zeta(s)| &= \int_{0}^s
(s -r )^{-\frac12 - H } \big| v(r) - v(s)\big| \,dr \leq M \|v\|_{\gamma}\int_{0}^s
\,s ^{-\frac12 - H + \gamma}ds\\
&\leq M \|v\|_{\gamma}\,s ^{\frac12 - H + \gamma} \;,
\end{equs}
so that the required bound follows again from Lemma~\ref{lem:vcondhold}.
For  $s  \ge 1$, we observe that $\zeta(s)$ has a   square integrable singularity, and thus
the boundedness of the control $v(\cdot)$ will be sufficient. Indeed it follows from the definition of $h$
\eref{eqn:hfun} that
\begin{equ}
|\zeta(s)| \leq  \int_{0}^{3/4}
(s -3/4)^{-\frac12 - H } \big| v(r) - v(s)\big| \,dr
 \leq  M \|v\|_{\gamma}
(s -3/4)^{-\frac12 - H } \;.
\end{equ}
Since this expression is square integrable, it follows that
$\tE \int_{0}^{\infty} |\zeta^2(s)| ds < \tp$, and the claim follows from  \eref{eqn:integsmall}.
\end{proof}

We conclude with the corresponding bound on the Malliavin derivatives of the control $v$:

\begin{proposition} \label{prop:sectermthm}Let  $v$ be as defined
in \eref{eqn:stochcont}, the operator $\mathcal{D}$ as defined in \eref{e:frac} and $\D$ be the Malliavin derivative.  Then
following estimate holds:
\begin{equ}[e:bound]
\tE\Big( \int_{0}^{\infty} \int_{0}^{\infty}
\|\D_t\tilde v(s)\|^2\,ds\,dt \Big) <  \tp\;.
\end{equ}
\end{proposition}
\begin{proof} The main difficulty  here is to obtain control over the 
modulus of continuity of the  quantity
$\D_s v(t)$ for all $s,t \in [0,1]$.
On route to establishing the H\"older continuity of the Malliavin derivative applied to
the stochastic control, we need to perform several
careful estimates using the fractional derivatives. 

First notice that, since $\tilde v(s) \in \mathcal{F}_1$ for every $s>0$ 
($\mathcal{F}_t, t \geq 0$ denotes the filtration generated by the increments of $W$), 
we have $\D_t\tilde  v(s)\big) = 0$ for $t > 1$, so that the integrand in \eref{e:bound} vanishes outside the set $[0,1]\times[0,1]$. 
Furthermore, it is straightforward to check that $\D_t$ and $\CD^{H-{1\over 2}}$ commute, so that 
\begin{equ}
\tE\Big( \int_{0}^{\infty} \int_{0}^{\infty}
\|\D_t\tilde v(s)\|^2\,ds\,dt \Big) 
 \leq \tE\Big( \int_{0}^{1} \int_{0}^{1}
\|\big(\CD^{H- \frac12} \D_tv\big)(s) \|^2 \,dt\,ds\Big)\;.
\end{equ}
By Lemma \ref{lem:Dalpop}, the claim follows if we can show that 
\begin{equ}
\tE\Big( \int_{0}^{1} \int_{0}^{1}
\Big|\int_{0}^s (s -r)^{-\frac{1}{2} - H} \big(\D_tv(s) - \D_tv(r)\big ) \,dr  \Big|^2 \,dt ds\Big)
< \tp\;. 
\label{eqn:fracfin}
\end{equ}
Expanding the integrand, we obtain
\begin{equs}
\Big| \D_t\Big(\int_{0}^s (s -r)^{-\frac{1}{2} - H} &\big(v(s) - v(r)\big ) \,dr \Big) \Big|
\leq  
 \int_{0}^t (s -r)^{-\frac{1}{2}- H} |\D_t v(s)|  \,dr   \\
&  \quad + \Big| \D_t\Big(\int_{t}^s (s -r)^{-\frac{1}{2} - H} \big(v(s) - v(r)\big ) \,dr \Big) \Big|\\
&\leq |s ^{\frac{1}{2} - H} - (s-t)^{\frac{1}{2}- H}|\,\|\D_t v\|_{0,1,\infty}  \\
&\quad + \Big| \int_{t}^s (s -r)^{-\frac{1}{2} - H} \D_t\big(v(s) - v(r)\big ) \,dr\Big| \;.  \label{eqn:mallfracint}
\end{equs}
Since the double integral of $|s ^{\frac{1}{2} - H} - (s-t)^{\frac{1}{2}- H}|^2$ is finite, the first term will be bounded
by $\tp$ if we can show that $\tE \sup_{t \in [0,1]} \|\D_t v\|_{0,1,\infty}^2 < \tp$.
Since this bound will be a byproduct of the analysis of the second term, we postpone it to \eref{eqn:supDtys} below.

The second term of  equation \eref{eqn:mallfracint} needs more delicate calculations.
We will show that there exists a random variable $K$ and an exponent $\alpha < 1+\gamma-H$
such that the following bounds hold:
\begin{equ}[e:boundHolder]
|\D_t v(s) - \D_t v(r)| \leq K |s-r|^\gamma |r-t|^{-\alpha} \;,\quad\text{with}\quad \tE\,K^2 < \tp\;.
\end{equ}
Assuming these bounds, we proceed similarly to the proof of Proposition~\ref{prop:Holdertrick}:
\begin{equs}
\tE\Big(\int_{0}^{1} \int_{0}^{1} 
&\Big| \int_{t}^s (s -r)^{-\frac{1}{2} - H}\D_t\big(v(s) - v(r)\big ) \,dr  \Big|^2 \, dt\, ds \Big) \\
&\leq \tE\Big( K^2 \int_{0}^{1} \int_{0}^{1} \Big| \int_{t}^s (s -r)^{-\frac{1}{2} - H + \gamma} (r-t)^{-\alpha}
dr \Big|^2 \, dt\, ds \Big) \\
&\leq \tE\Big( K^2 \int_{0}^{1} \int_{0}^{1} (s-t)^{1 +2\gamma - 2H - 2\alpha} \, dt\, ds \Big) 
< \tp\;,
\end{equs}
where we used the bound on $\alpha$ to obtain the last inequality. It therefore remains to show \eref{e:boundHolder}.

The Malliavin derivative of $v$ is given by
\begin{equ}
\D_t v(s) = h(s)\, \D_t  \Big(V(X_s)^*(J_{0,s}^{-1})^* (\A h \A^*)^{-1}\xi \Big)\;.
\end{equ}
Applying the product rule, we break this into a sum of three terms:
\begin{equs} 
\D_t v(s) &= h(s)\Big[\, \D_t  \Big(V(X_s)^*\Big)(J_{0,s}^{-1})^* (\A h \A^*)^{-1}\xi  +
V(X_s)^*\D_t  \Big((J_{0,s}^{-1})^*\Big) (\A h \A^*)^{-1}\xi  \\
&\hspace{3cm} +V(X_s)^*(J_{0,s}^{-1})^*\D_t  \big( (\A h \A^*)^{-1}\xi\big)\Big]\\
&\eqdef \I_1 + \I_2 + \I_3\;. \label{eqn:Mallsc}
\end{equs}
%We first observe that one has the following bound, the proof of which is postponed to the appendix:
%\begin{lemma} \label{lem:Dtthirdtm}
%For any $p \geq 1$, $\tE\left \|\D_t   (\A h \A^*)^{-1}\xi \right \|^p_{0,1,\infty} < \tp$.
%\end{lemma}

%We first handle $\I_3$. Since the process $X_t$ 
%is $\gamma$-H\"older continuous for any $\gamma < H$, if follows that
%the stochastic process $V(X_s)$ is also $\gamma$-H\"older continuous and we have
%\begin{equs}
%\big |V(X_s)^*(J_{0,s}^{-1})^* - V(X_r)^*(J_{0,r}^{-1})^*\big| &
%\leq M \|X\|_{0,1,\gamma} \|J_{0,\cdot}\|_{0,1,\gamma}|s -r |^{\gamma }\;.
%\label{eqn:estthird}
%\end{equs}
%Thus by Lemmas \ref{lem:Jmom}, \ref{lem:Dtthirdtm} and the estimate
%\eref{eqn:estthird} we obtain \eref{e:boundHolder} (with $\alpha = 0$) for the term $\I_3$ as required.
 
To make the proof compact,
 we introduce the stochastic process $Y_s$ given by collecting the various objects appearing in this expression:
\begin{equs}
Y_s = \big(X_s, V(X_s), (J_{0,s}^{-1})^*\big) \;.
\label{eqn:yt}
\end{equs}
It the follows from \eref{e:SDEmain} and the expression \eref{eqn:sflow} for the Jacobian $J_{0,t}$ that the components of $Y_s$ 
solve equations of the form \eref{e:nuhubounds}.
One can check that both the Jacobian $\mathcal{J}_{0,s}$ for $Y$ and its inverse $\CJ_{0,s}^{-1}$ then satisfy equations 
of the type \eref{e:zhat}. So by Corollary \ref{cor:secondvariation}, we obtain
\begin{equ}[e:boundsYJ]
\|Y\|_{\gamma} + \|\CJ_{0,\cdot}\|_{\gamma}+ \|\CJ_{0,\cdot}^{-1}\|_{\gamma} \le K(x_0) \exp\bigl(M \|\tilde B + \CG \omega\|_{\gamma}^{1/\gamma}\bigr)\;,
\end{equ}
for some constant $M$ and some continuous function $K$.

Furthermore, just as in \eref{e:MalliavinDer}, the Malliavin derivative of $Y$ is given by
\begin{equs}
\D_t Y_s &= c_H \int_{t}^s \mathcal{J}_{a,s} \,U(Y_a)\, (a -t)^{H-3/2} \,da \, 1_{t \leq s}\;,
\label{eqn:Dtys}
\end{equs}
for some smooth function $U$ which grows linearly at infinity.

We now tackle the bounds on the two very similar terms $\I_1$, $\I_2$. 
Since $t^{H-{3\over 2}}$ is integrable at $0$ and since $U$ grows linearly at infinity, it follows immediately that
\begin{equ}[eqn:supDtys]
\sup_{s,t \in [0,1]} |\D_t Y_s| \le M \|\CJ_{0,\cdot}\|_{\gamma}\|\CJ^{-1}_{0,\cdot}\|_{\gamma} \|Y_{\cdot}\|_{\gamma}\;.
\end{equ}
Thus by \eref{e:boundsYJ} we have $\tE \sup_{t \in [0,1]} \|\D_t v\|_{0,1,\infty}^2 < \tp$.
Regarding its modulus of continuity, note that 
\begin{equ}
\D_t Y_s - \D_t Y_r = \int_r^s \CJ_{a,s} U(Y_a) (a-t)^{H-{3\over 2}}\,da + \int_t^r \bigl(\CJ_{a,s}-\CJ_{a,r}\bigr) U(Y_a) (a-t)^{H-{3\over 2}}\,da\;.
\end{equ}
The first term is bounded by
\begin{equ}
M|r-t|^{H-{3\over 2}} |r-s| \|\CJ_{0,\cdot}\|_{\gamma}\|\CJ^{-1}_{0,\cdot}\|_{\gamma} \|Y_{\cdot}\|_{\gamma}\;,
\end{equ}
and the second term is bounded by
\begin{equ}
M |r-s|^\gamma \|\CJ_{0,\cdot}\|_{\gamma}\|\CJ^{-1}_{0,\cdot}\|_{\gamma} \|Y_{\cdot}\|_\gamma\;,
\end{equ}
so that there does exist a constant $M$ such that
\begin{equ}
|\D_t Y_s - \D_t Y_r | \le M|r-t|^{H-{3\over 2}} |r-s| \|\CJ_{0,\cdot}\|_{\gamma}\|\CJ^{-1}_{0,\cdot}\|_{\gamma} \|Y_{\cdot}\|_\gamma\;,
\end{equ}
uniformly over $t \le r\le s \le 1$.
Since $\gamma > {1\over 2}$, we do indeed have ${3\over 2} - H < 1 + \gamma - H$, so that this bound
is of the type \eref{e:boundHolder} with $\tE K^p < \tp$ for every $p>0$.

Notice furthemore that by Theorem \ref{thm:mallmatest}, 
the matrix $C_1 =\A h \A^* \in \mathbb{R}^{n^2}$ is almost surely
invertible and $\tE|\A h \A^*|^{-p} < \tp$ for all $p \geq 1$.
Combining this with \eref{e:boundsYJ} and the expressions
for $\CI_1$ and $\CI_2$, we conclude that \eref{e:boundHolder} does indeed hold for these terms.

Let us now turn to the term $\CI_3$.
Recall that the Fr\'echet  derivative of the inverse an  $n \times n$ matrix $ A $ in
the direction $H$ is given by $ D^HA^{-1} = -A^{-1} H A^{-1}$.
It thus follows from the chain rule that
\begin{equs}
\D_t  \big( (\A h \A^*)^{-1}\xi\big)  = -(\A h \A^*)^{-1} \,\bigl(\D_t   (\A h \A^*)\bigr)\, (\A h \A^*)^{-1}
\xi
\end{equs}
Thus by Cauchy-Schwartz we have,
\begin{equs}
\Big(\tE\, \left \|\D_t  \big( (\A h \A^*)^{-1}\xi\big) \right \|^p_{0,1,\infty}\Big)^2 &\leq 
\tE| (\A h \A^*)^{-1}|^{4p} \,\tE\,\|\D_t  \, (\A h \A^*) \|^{2p}_{0,1,\infty} |\xi|^2\;.
\end{equs}
By the product rule and the definition of $\A$,
\begin{equ}
\D_t  (\A h \A^*) = \int_{0}^{1} \big(\D_t J_{0,s}^{-1}\big)\,V(X_s) \,h(s)\,V^*(X_s) \,(J_{0,s}^{-1})^*\,ds
+ \ldots \;, \label{eqn:3otherts}
\end{equ}
where ``$\ldots$'' stands for the remaining three terms that are very similar. 
Thus from \eref{eqn:supDtys} it follows that
\begin{equs}
\tE\Big |\int_{0}^{1} \big(\D_t J_{0,s}^{-1}\big)\,&V(X_s) \,h(s)\,V^*(X_s) \,(J_{0,s}^{-1})^*\,ds \Big |^p \\
 &\le M\tE
\|\mathcal{J}_{0,\cdot}\|^p_{\gamma}\,\|\mathcal{J}^{-1}_{0,\cdot}\|^p_{\gamma} \,\|Y_\cdot\|^{2p}_{\gamma}
\, \|V(X_\cdot)\|^{2p}_{\gamma} < \tp.
\end{equs}
Since $J^{-1}$ and $V(X)$ play indistinguishable roles in \eref{eqn:yt} and \eref{e:boundsYJ}, the bounds on the 
remaining three terms follow in an identical fashion, so that
\begin{equ}
\E \sup_{t \in [0,1]} \bigl|\D_t  \big( (\A h \A^*)^{-1}\xi\big)\bigr|^p < \tp\;.
\end{equ}
Combining this with \eref{e:boundsYJ} and the definition of $\CI_3$, we finally conclude that 
 \eref{e:boundHolder} also holds for the term $\CI_3$ (this time with $\alpha = 0$), thus concluding the proof.
\end{proof}

\appendix
\section{Some technical lemmas}

We give the proofs of  some technical lemmas in this appendix. First, we show that Lemma~\ref{lem:divct}
holds, which we reformulate here:

\begin{lemma} 
Let $\omega  \in \CW$. Then, on
any time interval bounded away from $0$, the map
$t \mapsto \CG \omega(t)$  is $\CC^\infty$.
Furthermore, if we set $f_\omega (t) = t {d\over dt} \CG \omega(t)$, then we have $f_\omega(0) = 0$ and for every $T>0$, there exists
a constant $M_T$ such that 
$\|f_\omega \|_{0,T,\gamma} < M_T\|\omega\|_{(\gamma,\delta)}$.
\end{lemma}

\begin{proof}
The proof follows from the calculations done in
\cite{Hair:Ohas:07}.  Since the function $g(\cdot)$ defined  in \eref{eqn:gfun}
is smooth everywhere except at $0$, it follows from \eref{eqn:opera} that 
$m_t$ is smooth on any interval bounded away from $0$. Now
a simple change of variables in \eref{eqn:opera} yields
\begin{equ} [eqn:tctform]
t \,{dm_{t} \over dt}  = \gamma_H\int_0^{\infty} {1\over r }t\,g'(t)\omega\Bigl(-{t \over r}\Bigr)\,dr\;.
\end{equ}
Since $g'(t)$ is smooth everywhere except at $0$, 
for the map $ t \mapsto t\, {dm \over dt}(t)$ to have finite $\gamma$-H\"older norm
on $[0,T]$,
by \eref{eqn:tctform} we need to only check the regularity of $t \,g'(t)$
at $0$ and at infinity.  Indeed, if we show that  the function  $t \, g'(t)$ behaves like
$\mathcal{O}(t)$ for $t \ll 1$, and $\mathcal{O}(t^{H-1/2})$ for $t \gg 1$,
the conclusion follows from Lemma A.1 and Proposition A.2 of 
\cite{Hair:Ohas:07}. We have,
\begin{equs}[eqn:tgt]
t \,g'(t) &= t^{H-1/2} + (H - 3/2)\,t\,\int_0^1 {(u+t)^{H-5/2}\over
(1-u)^{H-1/2}}\,du +  M\,t^2\,\int_0^1 {(u+t)^{H-7/2}\over
(1-u)^{H-1/2}}\,du \\
&=g(t) + M\,t^2\,\int_0^1 {(u+t)^{H-7/2}\over
(1-u)^{H-1/2}}\,du 
\end{equs}
for a  constant $M$.  Notice that by Lemma A.1 of 
\cite{Hair:Ohas:07} we have  that $g(t)$ behaves like
$\mathcal{O}(t)$ for $t \ll 1$ and $\mathcal{O}(t^{H-1/2})$ for $t \gg 1$.
It can be seen easily that for $ t \gg 1$, the term
$ t^2\,\int_0^1 {(u+t)^{H-7/2}\over
(1-u)^{H-1/2}}\,du$ behaves like $\mathcal{O}(t^{H - 3/2})$. 
For $t \ll 1$, we have 
\begin{equs}
 t^2\,\int_0^1 {(u+t)^{H-7/2}\over
(1-u)^{H-1/2}}\,du &<  t\,\int_0^1{(u+t)^{H-5/2}\over
(1-u)^{H-1/2}}\,du \\
&<  t\,\int_0^1 u^{H-5/2} \Big(1 -
(1-u)^{1/2 - H}\Big)\,du + M(t)\, \label{eqn:tgtf2}
\end{equs}
where $M(t)$ is $\mathcal{O}(t)$ for $t \ll 1$. Now since
$(1 -(1-u)^{1/2 - H})$ is $\mathcal{O}(u)$ for $u \approx 0$, and
diverges like $\mathcal{O}((1-u)^{1/2 - H})$ for $u \approx 1$, the integral in
\eref{eqn:tgtf2} is finite. Thus from \eref{eqn:tgt} it follows 
that $t g'(t)$ is $\mathcal{O}(t)$ for $t \ll 1$ and we are done. 
\end{proof}

Our next technical results yields a bound on the H\"older constant of a function
defined on a large interval as a function of the H\"older constants of its restriction to
smaller subintervals:

\begin{lemma}\label{lem:holdgamadd}
For any $N \ge 1$ and any sequence of times $0 = u_0 < u_1<u_2 < \cdots <u_{N-1} < u_N = T$, 
the bound
\begin{equ}\label{eqn:holdbd02}
\|x\|_{0,T,\gamma} \leq  N^{1- \gamma} \, \max_{0 \leq i \leq N-1} \|x\|_{u_i,u_{i+1},\gamma} 
\end{equ}
holds.
\end{lemma}
\begin{proof}
We proceed by induction. For $N=1$, the bound \eqref{eqn:holdbd02} holds trivially.
Now assume \eqref{eqn:holdbd02} holds for some $N  \in \mathbb{N}$
and consider the sequence $0 = u_0 < u_1<u_2 < \cdots <u_{N-1} < u_N < 
u_{N+1} = T$.
 If the supremum $S\eqdef \sup_{s \neq t \in [0,T]} {|x(t)  - x(s)| \over |t-s|^\gamma}$
is attained either for $0 \leq s < t < u_{N}$ or for $u_1 \leq s < t < u_{N+1}$ then we can use the induction hypothesis
to conclude that
\begin{equs}
\|x\|_{0,T,\gamma} &= \|x\|_{0,u_N,\gamma} \vee \|x\|_{u_1,u_{N+1},\gamma} \\
&\leq  N^{1- \gamma} \, \max_{0 \leq i \leq N} \|x\|_{u_i,u_{i+1},\gamma}
\leq (N+1)^{1-\gamma}\, \max_{0 \leq i \leq {N}} \|x\|_{u_i,u_{i+1},\gamma}\;,
\end{equs} 
and we are done.
It remains to consider the case where
the supremum $S$ is attained for some $s \in [0,u_1]$ and $t \in [u_{N},T]$. 
We then have
\begin{equs}
S &\leq 
 {|x(u_1)  - x(s)| + |x(u_2)  - x(u_1)| + \cdots + |x(t)-x(u_{N})| \over |t-s|^\gamma}  \\
 &\leq 
   {|u_1-s|^\gamma \|x\|_{0,u_1,\gamma} \over |t-s|^\gamma} + {|u_2-u_1|^\gamma \|x\|_{u_1,u_2,\gamma} \over |t-s|^\gamma}
  + \cdots + {|t-u_N|^\gamma \|x\|_{u_{N},u_{N+1},\gamma} \over |t-s|^\gamma}\\
   &\leq \Big({|u_1-s|^\gamma \over |t-s|^\gamma} + {|u_2-u_1|^\gamma  \over |t-s|^\gamma} + 
   \cdots {|t-u_N|^\gamma  \over |t-s|^\gamma}\Big)\max_{0 \leq i \leq N} \|x\|_{u_i,u_{i+1},\gamma}\;.\label{e:boundS}
   \end{equs}
Since the function $x \mapsto x^\gamma$ is concave for $\gamma < 1$ and since $|u_1-s| + \ldots + |t-u_N| = |t-s|$, we have
\begin{equs}
{1 \over N+1} &\Big( {|u_1-s|^\gamma \over |t-s|^\gamma} + {|u_2-u_1|^\gamma  \over |t-s|^\gamma} + 
 \cdots {|t-u_N|^\gamma  \over |t-s|^\gamma}\Big)
 \leq \\ 
 &\Big( {1 \over N+1}{|u_1-s| \over |t-s|} + {1 \over N+1}{|u_2-u_1|  \over |t-s|} + 
   \cdots {1 \over N+1}{|t-u_N|  \over |t-s|}\Big)^{\gamma} = {1\over {(N+1)^\gamma}}.
 \end{equs}
Combining this with \eref{e:boundS}, we obtain
\begin{equ}
S \leq (N+1)^{1-\gamma}\, \max_{0 \leq i \leq {N}} \|x\|_{u_i,u_{i+1},\gamma}\;,
\end{equ}
thus concluding the proof.
\end{proof}

Finally, we provide the following interpolation result:

\begin{lemma} \label{lem:inter}
The bound
\begin{equs}
\|f\|_{0,T, \infty} \leq 2\max \Big(T^{-1/2}\|f\|_{L^2[0,T]} ,
\,\|f\|_{L^2[0,T]}^{{2\gamma \over 2\gamma +1 }} \,
\|f\|_{0,T,\gamma}^{{1 \over 2\gamma +1}} \Big)\;,
\end{equs}
holds for every $\gamma \in (0,1]$ and every $\gamma$-H\"older continuous function $f : [0, T ] \mapsto \R$. 
\end{lemma}

\begin{proof}
Let $f_{\mn} = \inf_{t \in [0,T]} |f(t)|$, so that
$f_{\mn} \leq  \frac{\|f\|_{L^2[0,T]}}{\sqrt{T}}$.  
Therefore it follows from the definition of H\"older continuity that  
\begin{equ}
\|f\|_{0,T,\infty} \leq f_\mn +  \|f\|_{0,T,\gamma}\,T^{\gamma} \leq \frac{\|f\|_{L^2[0,T]}}{\sqrt{T}} + \|f\|_{0,T,\gamma}\,T^{\gamma}\;.
\label{eqn:holdeq}
\end{equ} 
For $a>0$ define the set
\begin{equs}
L_f^a = \{ x \in [0,T]; |f(x)| > a \}.
\end{equs}
By Chebyschev's inequality we have, for any $a >0$, $
\lambda(L_f^a) \leq { \|f\|^2_{L^2[0,T]} \over a^2 }$
where $\lambda$ denotes the Lebesgue measure.
Thus for any $a > f_{\mn}$ and $x \in L_f^a$, we have
\begin{equ}
|f(x)| \leq a + \|f\|_{0,T,\gamma}\, |\lambda(L_f^a)|^{\gamma}\;,
\end{equ}
so that
\begin{equ}[eqn:holdbd2]
\|f\|_{0,T,\infty} \leq a + a^{-2\gamma}\|f\|_{0,T,\gamma}\|f\|^{2\gamma}_{L^2[0,T]}\;.
\end{equ}
Now we combine the estimates obtained in \eref{eqn:holdeq} and \eref{eqn:holdbd2}.
 Making the choice $a = \|f\|^{{1 \over 2\gamma +1}}_{0,T,\gamma}\|f\|^{{2\gamma  \over 2\gamma +1 }}_{L^2[0,T]}$
 in \eref{eqn:holdbd2} yields the bound
\begin{equ}[eqn:holdbd3]
\|f\|_{0,T,\infty} \leq 2\|f\|^{{1 \over 2\gamma +1}}_{0,T,\gamma}\|f\|^{{2\gamma  \over 2\gamma +1 }}_{L^2[0,T]}\;,
\end{equ}
provided that we have $ \|f\|^{{1 \over 2\gamma +1}}_{0,T,\gamma}\|f\|^{{2\gamma  \over 2\gamma +1 }}_{L^2[0,T]} \geq
 \frac{\|f\|_{L^2[0,T]}}{\sqrt{T}} \geq
f_{\mn}$. This is the case when $ \|f\|_{0,T,\gamma} \geq  \|f\|_{L^2[0,T]}T^{-\gamma -{1\over 2}}$ .
When $ \|f\|_{0,T,\gamma} \leq  \|f\|_{L^2[0,T]}T^{-\gamma -{1\over 2}}$, from \eref{eqn:holdeq} 
we obtain
\begin{equ}[eqn:holdbd4]
\|f\|_{0,T,\infty} \leq 2\|f\|_{L^2[0,T]} T^{-1/2}\;,
\end{equ}
thus completing the proof.
\end{proof}

\endappendix

\bibliographystyle{Martin}
\bibliography{./fbm}
\end{document}